%% file: arxiv.tex
\documentclass[11pt]{article}
\usepackage{fullpage}
\usepackage{natbib}
\usepackage{graphicx} 

\usepackage[utf8]{inputenc} 
\usepackage[T1]{fontenc}    
\usepackage[colorlinks,allcolors=blue,backref=page]{hyperref}     
\renewcommand*{\backref}[1]{}
\renewcommand*{\backrefalt}[4]{%
    \ifcase #1 (Not cited.)%
    \or        (Cited on page~#2.)%
    \else      (Cited on pages~#2.)%
    \fi}
\usepackage{url}            
\usepackage{booktabs}       
\usepackage{amsfonts}       
\usepackage{nicefrac}       
\usepackage{microtype}      
\usepackage{xcolor}         

\usepackage{amsthm}
\usepackage{amsmath}
\usepackage{cleveref}

\usepackage{amsfonts}
\usepackage{amssymb}
\usepackage{xpatch}
\usepackage{enumitem}

\usepackage{algorithm}
\usepackage{algpseudocode}

\usepackage{graphicx}
\usepackage{thmtools,thm-restate}
\usepackage{pifont}%

\usepackage{multirow}
\usepackage{wrapfig}
\usepackage{tikz}
\tikzstyle{vertex}=[circle, draw, inner sep=0pt, minimum size=6pt]
\usetikzlibrary{arrows}

\usepackage{xcolor}
\usepackage{colortbl} 
\usepackage{tabularx} 
\usepackage{makecell}  
\usepackage{tabstackengine} 
\usepackage{adjustbox} 
\usepackage{footnote}
\usepackage{caption}
\makesavenoteenv{tabular}
\usepackage{titletoc} 
\usepackage{comment}

\newtheorem{theorem}{Theorem}
\newtheorem{lemma}{Lemma}
\newtheorem{remark}{Remark}

\newtheorem{definition}{Definition}

\usepackage{thm-restate} 

\DeclareMathOperator\AndKeyWord{and}

   \DeclareMathOperator*{\argmin}{argmin}

\DeclareMathOperator\supp{supp}
\DeclareMathOperator\prog{prog}
\DeclareMathOperator\prev{prev}
\usepackage{thm-restate}

\usepackage{algpseudocode}
\usepackage{chngcntr} 
\usepackage[top=1in, right=1in, left=1in, bottom=1in]{geometry}

\input{macros}

\title{On the Limits of Biased Derivative Information for\\Nonconvex Stochastic Optimization}

\author{
\begin{tabular}{c@{\hspace{2cm}}c}
Anant Shyam & Brian Bullins \\
Purdue University & Purdue University \\
\texttt{shyama@purdue.edu} & \texttt{bbullins@purdue.edu}
\end{tabular}
}
\date{}
\begin{document}

\maketitle
\begin{abstract}
We consider the problem of finding $\delta$-stationary points for $\delta > 0$, i.e., $x \in \mathbb{R}^d$ such that $||\nabla F(x)|| \le \delta$, for smooth, non-convex objectives, where the derivative oracles are not only stochastic but also biased. In the first-order setting, we provide tight lower bounds for finding an $O((\epsilon + B^2)^{1/2})$-stationary point, for $\epsilon > 0$ and where $B$ is a bound on the gradient bias, matching the upper bounds of Ajalloeian and Stich (2020). We then establish bias-dependent lower bounds for algorithms that use higher-order derivative information for finding $O(\epsilon + B_{\max})$-stationary points, where $B_{\max}$ is a bound on the maximum bias for all derivatives. To complement these lower bounds, we develop trust-region based methods that, for certain ranges of bias, provide guarantees that match the corresponding lower bounds. We further improve upon the oracle complexity in high bias settings through a higher-order variance reduction scheme, in particular demonstrating the benefits, in some cases, of using higher-order derivative information, whereas such improvements are known to be unattainable for stochastic unbiased settings. 
\end{abstract}

\section{Introduction}
For a function $F: \mathbb{R}^d \rightarrow \mathbb{R}$ 
which has Lipschitz continuous first- and high-order derivatives, and bounded suboptimality $\Delta$ such that $F(0) - \inf_{x \in \mathbb{R}^d} F(x) \le \Delta$, we focus on the task of finding an $\epsilon$-stationary point: that is, $x \in \mathbb{R}^d$ such that, for $\epsilon > 0$, \begin{flalign*}
    &\norm{\grad F(x)} \le \epsilon.
\end{flalign*}
The problem of finding stationary points has been explored in numerous previous works for both convex~\cite{nesterov2012make,allen2018make} and non-convex~\cite{dimensionfreeacelerationgdconvex,stochasticfirstandzerothorder,LowerBoundsI, LowerBoundsII} settings, as well as for the more general context of finding approximate local minima~\cite{findingapproxlocalminimafastergd,carmon2018accelerated}. 
\newline \newline 
Meanwhile, there has been another line of work~\cite{SecondOrderOptNonconvexStochastic, stochasticcubicreg, stochasticfirstandzerothorder, allen2018make, allen2018natasha, stochasticfirstandzerothorder, restartedsgd, lowerboundsnonconvexstochastic, stochasticcubicreg, spider, stochasticnestedvarreduction}  investigating, for stochastic settings, the oracle complexity of finding $\epsilon$-stationary points in expectation, i.e, $\mathbb{E}||\nabla F(x)|| \le \epsilon$. For such settings, which have been widely studied in machine learning contexts~\cite{adam, lion, liu2025muon, sophia, stacey, bernstein2018signsgd}, access to the objective function is restricted to stochastic estimates of its derivatives: for each query point $x$ and derivative order $i$, a stochastic $i^{th}$ order derivative oracle returns $\widehat{\nabla}^i F(x, \xi)$, where $\xi_i$ is random variable drawn from some distribution $P_\xi$ such that \begin{flalign*}
    &\mathbb{E}[\widehat{\nabla}^iF(x, \xi_i)] = \nabla^i F(x)  \hspace{0.5em} \AndKeyWord \hspace{0.5em}\mathbb{E}||\widehat{\nabla}^i F(x, \xi_i) - \nabla^iF(x)||_\mathrm{op}^2 \le \sigma_i^2. 
\end{flalign*}
For objectives with Lipschitz continuous gradients, one can achieve an $O(\epsilon^{-4})$ oracle complexity assuming access to only a stochastic gradient oracle~\cite{stochasticfirstandzerothorder}. With additional assumptions on the objective (e.g Lipschitz continuous Hessian, mean-squared smoothness) and the oracle access (i.e the ability to query a stochastic, second order derivative oracle), it is possible to achieve an oracle complexity of $O(\epsilon^{-3})$~\cite{spider, SecondOrderOptNonconvexStochastic}. Furthermore, Arjevani et al. \cite{SecondOrderOptNonconvexStochastic} showed that $O(\epsilon^{-3})$ was the optimal oracle complexity assuming access to not just stochastic second order oracles but stochastic higher derivative oracles for \emph{all} derivatives $p \ge 2$, as well as that the optimal complexity assuming just first-order stochastic oracle access was $O(\epsilon^{-4})$, thereby introducing an \emph{elbow effect}.
\newline \newline 
However, in many practical settings, it becomes necessary to relax these unbiased derivative assumptions, with such examples appearing in distributed \cite{dryden2016communication,aji2017sparse, alistarh2018convergence}, gradient-free \cite{nesterov2017random}, and bandit convex optimization \cite{hu2016bandit}, thus highlighting the need to better understand the limits of working with biased derivative information. In such settings, the objective function is further restricted to stochastic and biased estimates of its derivatives: for each query point $x$ and derivative order $i$, a stochastic and biased $i^{th}$ order derivative oracle returns $\tilde{\nabla}^i F(x, \xi_i, b_i)$, where $\xi_i$ is a random variable drawn from some distribution $P_\xi$ and $b_i$ represents the bias of the $i^{th}$ derivative estimate, such that \begin{flalign*}
    &\tilde{\nabla}^i F(x, \xi_i, b_i) = \nabla^i F(x) + \xi_i(x, z) + b_i(x)
\end{flalign*} where \begin{flalign*}
    &\mathbb{E}_{z \sim P_z}[\xi_i(x, z)] = 0 \hspace{0.5em} \AndKeyWord \hspace{0.5em} \mathbb{E}||\xi_i(x, z)||_\mathrm{op}^2 \le \sigma_i^2 \hspace{0.5em} \AndKeyWord \hspace{0.5em} ||b_i(x)||_\mathrm{op} \le B_i
\end{flalign*}
In this paper, we establish the limits of biased derivative information for both first- and higher-order settings. 
We first complement the first-order upper bound in \cite{OnConvergenceSGDBiased} by providing a matching lower bound. We then show how to handle biased and stochastic \emph{high-order} derivative information, through providing corresponding higher-order lower bounds and developing higher-order trust-region and variance-reduction based algorithms to complement these lower bounds. Moreover, we further show that, unlike in the \emph{unbiased} case \citep{SecondOrderOptNonconvexStochastic}, appealing to higher order information beyond second-order information \emph{can} offer benefits for certain ranges of bias. 
\subsection{Our Main Contributions}
Below, we outline the main contributions of this paper. 
\paragraph{First-Order Lower Bound.} We begin by considering in Section \ref{section: lowerbounds} the first order oracle complexity for finding a stationary point in expectation. 
In particular, letting $B_1$, $L_1$, $\Delta$ and $\sigma_1$ represent bounds on the gradient bias, Lipschitz constant, suboptimality, and gradient variance, respectively, we provide, in Theorem \ref{thm:firstorderlowerbound}, a lower bound of \[ \Omega\left(\frac{\Delta L_1}{\epsilon + B_1^2} + \frac{\Delta L_1 \sigma_1^2}{\epsilon^2 + B_1^4}\right) \] on the biased-stochastic-first-order query complexity for finding an $O((\epsilon + B_1^2)^{1/2})$ stationary point in expectation. As a consequence, our results imply that the upper bounds of Ajalloeian and Stich~\cite{OnConvergenceSGDBiased} for this setting are tight. 
\paragraph{Higher-Order Lower Bound.} Given these tight lower bounds in the first-order setting, it is natural to consider, as we address in Section~\ref{section: lowerbounds}, alternative approaches based on higher-order information. To this end, we establish, in Theorem \ref{thm:higherorderlowerbound}, the limits of algorithms that use derivative orders $p \ge 2$ for finding $O(\epsilon + B)$ stationary points in expectation. In addition, as a special case of our rates, we recover previous lower bounds in the unbiased setting~\cite{SecondOrderOptNonconvexStochastic}.
\paragraph{Minibatch Derivative Estimation.} Complementing the lower bounds above, we provide in Section~\ref{section: upperbounds} upper bounds on convergence rates for techniques based on higher-order derivative information. Specifically, we establish in Theorem~\ref{thm:upperbounddeterministic} high-order oracle complexities for finding approximate stationary points in the case where $B \ge \omega(1)$.
We derive similar bounds for low-bias regimes, and we further note that our lower bounds are tight for $B \le O(\epsilon)$. To achieve these bounds, we develop an algorithm in which each derivative estimate $D^i$ can be computed as an average of $n_i$ calls to the $i^{th}$-order biased and stochastic derivative oracle  At each step, we solve a regularized higher-order Taylor subproblem. 
\paragraph{Variance Reduction Based Derivative Estimation.} Given numerous prior works which show the advantages of using variance reduction in derivative estimation, we also utilize, in Section~\ref{section: upperbounds}, variance reduction based techniques in order to improve the oracle complexity bound for finding stationary points. In theorem \ref{thm:rvr-upperbound}, we derive an improved bound for the $B = \Theta(1)$ case for finding $O(\epsilon + B)$ stationary points in expectation, and derive a similar upper bound for the $B \ge \omega(1)$ case, though with slightly weaker stationary guarantees. To achieve these bounds, we devise a higher-order recursive variance reduction scheme to construct the derivative estimates, and like before, minimize the regularized higher-order Taylor model at each iteration. 
Notably, unlike the stochastic unbiased setting, here we do in fact realize benefits by appealing to higher-order derivative information for the constant bias setting. 
\subsection{Additional Related Works}
\paragraph{Biased Gradient Methods.} We briefly discuss some additional prior work related to biased gradient methods. In \cite{biasedstochasticfirstordermetalearning}, the authors proposed a biased SGD algorithm and analyzed the sample complexities for convex and nonconvex objectives. In \cite{stochgradmethodbiasedestimation}, the authors explore a balance between biased and unbiased estimation of the gradient to resolve the tradeoff between the cost and benefit of computing an unbiased derivative estimation. In \cite{biasedstochgradestimation, stochcontrollablebiasedoracles}, the authors also present algorithms under biased gradient estimation. In \cite{demidovich2024guide}, the authors present an analysis of biased SGD is convex and nonconvex settings, under weaker assumptions than many previous works. 
\paragraph{Deterministic Oracles.} 
Also of interest has been the problem of finding $\epsilon$-stationary points for smooth nonconvex objectives in the case of deterministic (i.e., noiseless and unbiased) oracles. First, an improved $O(\epsilon^{-\frac{7}{4}}) $ query complexity for first-order methods has been shown \cite{dimensionfreeacelerationgdconvex,li2023restarted} when assuming that the Hessian is also Lipschitz continuous, and this has recently been shown to be tight~\cite{zhou2026sharp} due to a new lower bound that improves upon the previous $\Omega(\epsilon^{-\frac{12}{7}})$ lower bound~\citep{LowerBoundsII}. Furthermore, when using $p^{th}$ order oracles (assuming that all $p$ derivatives are Lipschitz continuous), the optimal oracle complexity is $O(\epsilon^{-\frac{p+1}{p}})$~\cite{LowerBoundsI,birgin2017worst}. Moreover, in \cite{findingapproxlocalminimafastergd}, the authors present a method that uses a cubic regularized Newton step to achieve an $\tilde{O}(\epsilon^{-\frac{7}{4}})$ oracle complexity for finding an $\epsilon$-approximate local minimum, i.e., an $\epsilon$-stationary point $x$ that also satisfies $\nabla^2 f(x) \succeq -\sqrt{\epsilon} I$. 
\subsection{Paper Organization} 
We formally introduce our problem setup in Section \ref{section:setup_background}, including the function class and derivative oracle properties. In Section \ref{section: lowerbounds}, we present the lower bounds for both the first and higher order settings. In Section \ref{section: upperbounds}, we present algorithms that use minibatch-based derivative estimation (Algorithm \ref{alg:novariancereduction}) and variance reduction based derivative estimation (Algorithm \ref{alg:var-reduction-pth-order}). 
\newline \newline 
\textbf{Notation.} 
For some $1 \le i \leq p$, let $\nabla^i F$ refer to the $i^{th}$ derivative of a function $F \in \mathcal{C}^p$, where $\mathcal{C}^p$ denotes the set of $p$ times differentiable, continuous functions. For all $i$, $[\nabla^i F(x)]_{j_1, \ldots, j_i} = \frac{\partial^i F}{\partial x_{j_1} \ldots \partial x_{j_i}}$. For matrices $A$ and tensors $T$, $||\cdot ||_\mathrm{op}$ denotes the operator-2 norm, and unless otherwise specified, $||\cdot||$ refers to the Euclidean norm (if dealing with a particular $x \in \mathbb{R}^d$) or the operator-2 norm (if dealing with any matrices or tensors) as well. For a symmetric tensor $T$, we let $||T||_\mathrm{op} = \sup_{||v|| = 1} |\langle T, v, \ldots, v  \rangle |$. We also let $B = \max_{1 \le i \le p} B_i$. 
\section{Setup and Background}
\label{section:setup_background}
\subsection{Function Class}
We consider smooth,  differentiable functions in the following function class: 
\begin{flalign*}
    &\mathcal{F}_p(\Delta, L_{1: p}) = \{F: \mathbb{R}^d \rightarrow \mathbb{R}: ||\nabla^q F(x) - \nabla^q F(y)|| \le L_q ||x - y|| \hspace{0.25em} \forall x, y \in \mathbb{R}^d, q \in \{1, \ldots, p \} \}
\end{flalign*} where for all $F \in \mathcal{F}_p(\Delta, L_{1: p})$, we have that \begin{flalign*}
    &F(0) - \inf_{x \in \mathbb{R}^d} F(x) \le \Delta 
\end{flalign*} 
where $L_{1: p} = (L_1, \ldots, L_{p})$ represent the Lipschitz constants of the first to $p^{th}$ derivatives of $F$. 
\subsection{Oracles}
For such a function $F \in \mathcal{F}_p(\Delta, L_{1: p})$, we consider a class of biased and stochastic derivative oracles for $\nabla^1F, \ldots, \nabla^p F$ defined by a distribution $P_z$ over a measurable set $\mathcal{Z}$ and an estimator \begin{flalign*}
    &O_F^p(x, z, b) := (\tilde{F}(x, z, b_0), \tilde{\nabla} F(x, z, b_1), \ldots, \tilde{\nabla}^p F(x, z, b_p))
\end{flalign*} where $\tilde{\nabla}^qF(x, z, b_q) $ is a biased and stochastic estimate for $\nabla^q F(x)$. For all $x$ and $q \in [p]$, we have that $\tilde{\nabla}^q F(x, z, b) = \nabla F^q(x) + \xi_q(x, z) + b_q(x)$, where $\mathbb{E}_{z \sim P_z} [\xi_{q}(x, z)] = 0$,  $\mathbb{E}||\xi_{q}(x, z)||^2 \le \sigma_q^2$, and $||b_q(x)|| \le B_q$. Given variance parameters $\sigma_{1: p}$ and bias parameters $B_{1: p}$, we define the oracle class $\mathcal{O}_p(F, \sigma_{1: p}, B_{1: p})$ to be the set of all biased and stochastic $p^{th}$ order oracles such that the conditions above hold. 
\section{Lower Bounds}
\label{section: lowerbounds}
We first consider the scenario of finding an $O(f(\epsilon) + g(B_1, \ldots, B_p))$ stationary point in expectation, where $f$ and $g$ are positive functions of the precision parameter $\epsilon$ and the bias terms respectively. Previous work~\cite{OnConvergenceSGDBiased} has shown how to achieve an upper bound of 
 \begin{flalign*}
    &O\left(\frac{\Delta L_1}{\epsilon + B_1^2} + \frac{\Delta L_1\sigma_1^2}{\epsilon^2 + B_1^4}\right)
\end{flalign*} on the number of stochastic first-order oracle queries for finding iterates $\{x_t \}$ where $\frac{1}{T} \sum_{t = 0}^{T - 1} \mathbb{E}||\nabla F(x^{(t)})||^2 = O(\epsilon + B_1^2)$, which implies that $\min_t ||\nabla F(x^{(t)})|| \leq O((\epsilon + B_1^2)^{\frac{1}{2}})$.
Thus, we begin by establishing this oracle complexity is optimal by providing a matching lower bound, as presented in Theorem \ref{thm:firstorderlowerbound}. 
\begin{theorem}
    \label{thm:firstorderlowerbound} 
    When $p = 1$, there exists $F \in \mathcal{F}_1(\Delta, L_1)$ and $(O_F^1, P_z) \in \mathcal{O}_1(F, \sigma_1, B_1)$ such that for any first-order zero-respecting  algorithm (Definition \ref{definition:zerorespectingalg}) 
    where $\epsilon < \frac{1}{4}$ and $B_1 \le O(1)$, the minimum number of queries to obtain a $(\epsilon + B_1^2)^{\frac{1}{2}}$ stationary point with constant probability is bounded below by \begin{flalign*}
        &\Omega \left(\frac{\Delta L_1}{\epsilon + B_1^2} + \frac{\Delta L_1 \sigma_1^2}{\epsilon^2 + B_1^4}\right)
    \end{flalign*}
\end{theorem}
Given the matching upper and lower bounds in the first order setting, one natural question is whether one could observe analogous lower and upper bounds in higher order settings. In Theorem \ref{thm:higherorderlowerbound}, we derive a lower bound for finding an $O(\epsilon + B)$ stationary point using derivatives $1, \ldots, p$ for $B \le \frac{\sqrt{3}}{2} \sigma_1$. The function $F$ that we will use is a rescaling of the function $F_T$ later defined in eq.\eqref{eq:hardfunction}. In Lemma \ref{lemma:hardfunctionproperties}, we prove that for all $x, y \in \mathbb{R}^d$, there exists a constant $C \ge 0$ such that $||\nabla F_T(x)|| \le C \sqrt{T}$. Since we scale this function as the following $F_T^*(x) = \alpha F_T(\beta x)$ for constants $\alpha, \beta$, we have that $||\nabla F_T(x)|| \le C_1 \sqrt{T}$ for a different $C_1$. For $B \ge C_1 \sqrt{T} - \epsilon$, we need zero oracle queries to reach the given stationarity condition, as for any $x \in \mathbb{R}^d$, we have that $||\nabla F(x)|| \le C_1 \sqrt{T}$, so the lower bound is vacuous for this scenario. 
\begin{theorem}
    \label{thm:higherorderlowerbound}
    For all $p \ge 2$, $\Delta$, $L_{1: p}$, $\sigma_{1: p} > 0$, $\epsilon < \sqrt{\sigma_1}$, and $B \le \frac{\sqrt{3}}{2} \sigma_1$, there exists $F \in \mathcal{F}_p(\Delta, L_{1: p})$ and $(O_F^p, P_z) \in \mathcal{O}_p(F, \sigma_{1: p}, B_{1: p})$ such that for any $p^{th}$ order zero-respecting algorithm (Definition \ref{definition:zerorespectingalg}), the number of queries to obtain a point an $\epsilon + \max_{i} B_i$ stationary point with constant probability is bounded below by 
    \begin{flalign*}
        &\Omega(1) \cdot \frac{(\sigma_1^2 - 4(\epsilon + B)^2 B_1^2) \Delta}{32 (\epsilon + B)^3 \ell_0^2 \Delta_0} \cdot \\ 
    &\min_{q' \in \{1, \ldots, p \},  q \in \{2, \ldots, p \} } \min \left \{\left(\frac{\ell_0^2(\sigma_q^2 + B_q^2)}{2 \ell_{q - 1}^2(\sigma_1^2 + B_1^2 - 4(\epsilon + B)^2 B_1^2)}\right)^{\frac{1}{2(q - 1)}}, \left(\frac{\sigma_q^2 + B_q^2}{8(\epsilon + B)^2 B_q^2}\right)^{\frac{1}{2(q - 1)}}, \left(\frac{L_{q'}}{2 (\epsilon + B) \ell_{q'}}\right)^{\frac{1}{q'}} \right \} 
    \end{flalign*}
\end{theorem}

\begin{proof}[Proof  Sketch] Here, we provide a brief proof sketch, deferring the full proofs to Appendix \ref{appendix: lower_bound_proofs}. We outline the proof of Theorem \ref{thm:higherorderlowerbound}, as the proof of Theorem \ref{thm:firstorderlowerbound} will be a similar argument to this one. We first introduce the following ``hard'' function~\cite{LowerBoundsI} where for a fixed $T \ge 0$ and $x \in \mathbb{R}^T$: \begin{flalign}
        \label{eq:hardfunction}
        &F_T(x) = -\Psi(1) \Phi(1) + \sum_{i = 2}^T [\Psi(-x_{i - 1}) \Phi(-x_i) - \Psi(x_{i - 1}) \Phi(x_i)]
    \end{flalign} 
    where \begin{flalign*}
    &\Psi(x) = \begin{cases}
    0, & \text{if } x \le \frac{1}{2} \\
    \exp(1 - \frac{1}{(2x - 1)^2}),   & \text{if } x > \frac{1}{2}, 
\end{cases}, \Phi(x) = \sqrt{e} \int_{-\infty}^x e^{-\frac{1}{2}t^2} dt,
\end{flalign*} and we let $\ell_p$ represent the Lipschitz constant of the $p^{th}$ derivative of $F_T$. Inspired by~\cite{SecondOrderOptNonconvexStochastic}, we then show how to construct the following series of derivative estimators for this function (for all derivatives $i \in \{1, \ldots, p \}$): \begin{flalign*}
    &[\tilde{\nabla}^qF_T(x, z, b)]_i = (1 + \boldsymbol{1}\{i > \prog_{\frac{1}{4}}(x) \}(\frac{z}{\rho} - 1)) \cdot (\nabla_i^q F_T(x) + b_i^q(x))
\end{flalign*} which crucially account for the fact that the oracles are not only stochastic \emph{but also biased}. If $T$ is a $k$-dimensional tensor, then $T_i$ is a $k - 1$-dimensional subtensor where $[T_i]_{j_1, \ldots, j_{k - 1}} = T_{i, j_1, \ldots, j_k}$ and $\prog_\alpha(x) = \max \{i \ge 0, |x_i| > \alpha \}$, representing the highest index of $x \in \mathbb{R}^d$ that is at least $\alpha$ away from zero. We carefully devised the above construction such that this collection of derivative estimators forms a probability-$\rho$ zero-chain (for some $0 \le \rho \le 1$) where the following properties hold: \begin{flalign*}
    &\Pr(\exists x | \prog(\tilde{\nabla}^1 F(x, z, b), \ldots, \tilde{\nabla}^pF(x, z, b)) = \prog_{\frac{1}{4}}(x) + 1) \le \rho \\
    &\Pr(\exists x | \prog(\tilde{\nabla}^1 F(x, z, b), \ldots, \tilde{\nabla}^pF(x, z, b)) = \prog_{\frac{1}{4}}(x) + i) = 0
\end{flalign*} for all $i > 1$. Through this zero-chain construction (as is well studied in \cite{adil2026convex, SecondOrderOptNonconvexStochastic,LowerBoundsI, LowerBoundsII, lowerboundsnonconvexstochastic}), we enforce that every oracle query can reveal information about at most one new coordinate, thereby requiring any algorithm to make sequential progress, which yields a lower bound on the number of queries to make sufficient progress. (See Appendix \ref{appendix: lower_bound_proofs} for a more formal explanation.) Given that $\{\tilde{\nabla}^q F_T(x, z, b) \}$ form a probability-$\rho$ zero-chain, we used a scaled version of $F_T$, parametrized by two constants $\alpha$ and $\beta$: \begin{flalign*}
    &F_T^*(x) = \alpha F_T(\beta x)
\end{flalign*} where we solve for $\alpha$ and $\beta$ to enforce the required suboptimality, noise, bias, and higher order Lipschitz conditions. 
\end{proof}
\section{Upper Bounds}
\label{section: upperbounds}
\subsection{Minibatch-Based Derivative Estimation}
Given these lower bounds, we develop various algorithms to show how the upper bounds compare. In Algorithm \ref{alg:novariancereduction}, we minimize a regularized Taylor model of our objective $F$ at each iteration. In Theorem \ref{thm:upperbounddeterministic}, we prove an upper bound for the oracle complexity of Algorithm \ref{alg:novariancereduction} for reaching an $O(\epsilon + B)$ stationary point. 

\begin{theorem}
    \label{thm:upperbounddeterministic}
    For any function $F \in \mathcal{F}_p(\Delta, L_{1:p})$, where $p \ge 2$, $\epsilon > 0$, with biased and stochastic $p^{th}$-order oracles in $\mathcal{O}(F, \sigma_{1:p}, B_{1:p})$ where $B \ge  \Omega(\epsilon^{\frac{3p}{3p + 1}})$, with probability at least $\frac{5}{8}$, Algorithm \ref{alg:pth_deterministic} returns a point $\hat{x}$ such that $||\nabla F(\hat{x})|| \le O(\epsilon + B)$ and performs at most \[
        \begin{cases} 
              O\left(\frac{\Delta (\max_i \sigma_i)^2 }{\epsilon^3 (B + 1)^{\frac{p + 1}{p}}} + \frac{(\epsilon + B)^{\frac{p + 1}{p}}(\max_i \sigma_i)^2}{\epsilon^3 (B + 1)^{\frac{p + 1}{p}}}\right) & \Omega(\epsilon^{\frac{3p}{3p + 1}}) \le B \le \Theta(1) \\
              O\left(\frac{\Delta (\epsilon + B)^{\frac{p + 1}{p}}(\max_i \sigma_i)^2}{\epsilon^3 (B + 1)^{\frac{p + 1}{p}}}\right) & B \ge \omega(1) \\
        \end{cases}
        \]
     queries to the stochastic and biased derivative oracles. 
\end{theorem}

Here, the fact that 
we require $B \ge \Omega(\epsilon^{\frac{3p}{3p + 1}})$ is due to requirements of the algorithm in terms of batch size and that the existence of constants $\{C_i \}$ is guaranteed by Lemma \ref{lemma:constantsC_i}. Note in the unbiased case (i.e $B = 0$), we recover the $O(\epsilon^{-3})$ guarantee known from \cite{SecondOrderOptNonconvexStochastic}. 
Moreover, as $p \rightarrow \infty$, the oracle complexity worsens, thereby implying that using derivative information beyond the second order does not help in this scenario. Therefore, in the minibatch derivative estimation setting, setting $p = 2$ yields optimal oracle complexity as compared to any $p > 2$. Interestingly, however, when using variance reduction, we do realize benefits to appealing to higher order derivative information for certain ranges of bias.

\input{higher-order-no-variance-reduction}

\begin{proof}[Proof Sketch]
    Here, we provide a brief proof sketch of Theorem \ref{thm:upperbounddeterministic}, deferring the full proof to Appendix \ref{appendix_a:minibatch_deriv_proof}. In Lemma \ref{lemma:constantsC_i}, we prove that there do exist constants $C_i$ that satisfy the condition on line 1 of Algorithm \ref{alg:novariancereduction}. We then derive a lower bound for $F(x) - F(y)$ in Lemma \ref{lemma:diffxandy} for all $M \ge 8L_p$ and $0 \le \eta < 1$, where $x \in \mathbb{R}^d$ and $y \in \argmin_{z: ||z - x|| \le \eta} m_x(z)$, where $m_x$ represents the regularized $p^{th}$ order model of $F$ around $x$ using biased and stochastic derivatives: \begin{flalign*}
        &m_x(y) = F(x) + \sum_{i = 1}^p \frac{1}{i!} D^{(i)}[y - x]^i + \frac{M}{(p + 1)!} ||y - x||^{p + 1}
    \end{flalign*} We then extend this Lemma to the case where $D^{(i)}$ are random variables in Lemma \ref{lemma:randomvarderivatives}, and then (through using an intermediary Lemma \ref{lemma:gradprobbound})) prove an upper bound of $\frac{3}{8}$ on the probability of reaching a point $\hat{x}$ such that $\Pr(||\nabla F(\hat{x})|| \ge \frac{9M}{8p!} (\epsilon + B))$, which completes the proof. 
\end{proof}
Below, we also provide a comparison to the lower bound in Theorem \ref{thm:higherorderlowerbound} in terms of the oracle complexity.

\begin{table}[h]
\centering
\begin{tabular}{ccc}
\toprule Bias Regime & Lower Bound (Theorem \ref{thm:higherorderlowerbound}) & Upper Bound (Theorem \ref{thm:upperbounddeterministic}) \\
\midrule
$B \le O(\epsilon)$ & $\Omega(\epsilon^{-3})$ & $O(\epsilon^{-3})$ \\
$B = \Theta(\epsilon^{q}), 0 \le q < 1$ & $\Omega(\epsilon^{-3q})$ & $O(\epsilon^{-3})$ \\
$B \ge \omega(1)$ & Trivial & $O(\epsilon^{-3})$ \\
\bottomrule
\end{tabular}
\end{table} This demonstrates that the lower bound in Theorem \ref{thm:higherorderlowerbound}
is tight for $B \le O(\epsilon)$, which represents the low bias regime. For the high bias regime, where $B \ge \omega(1)$, we note that if one knew that $B$ could be expressed in this form, the one could directly return the starting iterate $x^{(0)}$. Therefore, in theory, one could match the lower bound for this setting as well. We do not assume such knowledge in our algorithm. We work on improving the upper bound for the intermediate bias setting through a variance reduction scheme (Algorithm \ref{alg:var-reduction-pth-order}), and succeed in improving upon the upper bound for the $B = \Theta(1)$ scenario to $O(\epsilon^{-2})$, and leave further improvements for future work. 
 
\subsection{Variance Reduction}
Here, we investigate the potential advantages of using variance reduction in a biased setting. Many previous works have primarily relied on recursive variance reduction (e.g. \cite{spider}) to compute cheap estimators of the gradient $\nabla F(x^{(t)})$. In our implementation of recursive variance reduction, we build on that of \cite{SecondOrderOptNonconvexStochastic} by estimating $\nabla^i F(x^{(t)}) - \nabla^i F(x^{(t + 1)})$ by averaging $\nabla^{i + 1}F$-vector products for all $i \in [p]$, instead of just doing this with the gradient. To derive this estimator, we first note that for all $i$, it holds that (by the Fundamental Theorem of Calculus) for all $x, x'$: $\nabla^iF(x) - \nabla^i F(x') = \int_{0}^1 \nabla^{i + 1} F(xt + x'(1 - t))(x - x') dt$. Now, to approximate this integral, we construct the following estimator for $\nabla^i F$,  where $K$ is chosen to be proportional to $||x - x'||^2$: \begin{flalign*}
    &\tilde{\nabla}^i F = \frac{1}{K} \sum_{k = 0}^{K - 1} \tilde{\nabla}^{i + 1} F(x \cdot (1 - \frac{k}{K}) + x'\cdot \frac{k}{K}, z^{(i)}, b_i)(x - x')
\end{flalign*} 
We reset the derivative estimators according to a defined probability metric $b$ and dynamically set the batch size proportional to the difference between the current iterate and the previous iterate squared and incorporate this recursive variance reduction approach for all $p$ derivatives.
 In Theorem \ref{thm:rvr-upperbound}, we analyze the oracle complexity of a variance-reduction based algorithm (Algorithm \ref{alg:var-reduction-pth-order}) for finding an $O(\epsilon + B)$ stationary point when $B = \Theta(1)$ and a $O((\epsilon^2 + B^2)^{\frac{1}{2}}(\epsilon + B))$ stationary point when $B \ge \omega(1)$. We note that for when $B \le O(\epsilon)$, algorithm \ref{alg:novariancereduction} is already optimal, so we focus on these intermediate and high bias regimes below. 
\begin{algorithm}[H]
\caption{\textbf{H}igher-\textbf{O}rder \textbf{R}ecursive \textbf{V}ariance \textbf{R}eduction  \textbf{(HO-RVR)}}
    \label{alg:rvr_pth_order}
\label{alg:sgd}
\begin{algorithmic}[1]
\Require Precision parameter $\epsilon$, probability $b$, current iterate $x$, previous iterate $x_{\prev}$, derivative order $i$, derivative estimate with respect to $x_{\prev}$, $D_{\prev}^i$, Biased and stochastic oracle $(O_F^p, P_z) \in \mathcal{O}_p(F, \sigma_{1: p}, B_{1: p})$ for $F \in \mathcal{F}_p(\Delta, L_{1: p})$
\vspace{0.25em}
\State Set 
            $K = \left \lceil \frac{5(\sigma_{i + 1}^2 + L_{i + 1} \epsilon)}{b \epsilon^2} \cdot ||x - x_{\prev}||^2\right \rceil$

\State Set $n = \left \lceil  \frac{5 \sigma_i^2}{\epsilon^2}\right \rceil$
\State Sample $C \sim $ Bernoulli$(b)$. 
\If{C is 1 \textbf{or} $D_{\prev}^i$ is None}
            \State Query the $i^{th}$ order oracle $n$ times at $x$ and set 
                $D^{(i)} = \frac{1}{n } \sum_{j = 1}^n \tilde{\nabla}^i F(x, z^{(j)}, b_i), \hspace{0.5em} z^{(j)} \sim P_z$

    \Else
        \State For $k \in \{0, \ldots, K \}$, set 
            $x^{(k)} = \frac{k}{K} x + (1 - \frac{k}{K})x_{\prev}$

        \State Query the $i^{th}$ order oracle at the points $\{x^{(k)} \}_{k = 0}^{K - 1}$ and set \begin{align*}
            &D^{(i)} = D^{(i)}_{\prev} + \sum_{k = 1}^K \tilde{\nabla}^{i + 1} F(x^{(k - 1)}, z^{(k)}, b_{i + 1}), \hspace{0.5em} z^{(k)} \sim P_z
        \end{align*}
\EndIf
\State \Return $D^{(i)}$
\end{algorithmic}
\end{algorithm}

\begin{algorithm}[H]
\caption{\textbf{H}igher-\textbf{O}rder \textbf{R}ecursive \textbf{V}ariance  \textbf{R}eduction \textbf{D}erivative Estimation
\textbf{(HO-RVR-D)}}\label{alg:var-reduction-pth-order}
\begin{algorithmic}[1]
    \renewcommand{\algorithmicrequire}{\textbf{Input:}} \Require Precision parameter $\epsilon$, Biased and stochastic oracle $(O_F^p, P_z) \in \mathcal{O}_p(F, \sigma_{1: p}, B_{1: p})$ for $F \in \mathcal{F}_p(\Delta, L_{1: p})$, derivative order $p$ .
\vspace{0.5em}
\State Pick $b$ such that $0 < b \le 1$, let $B = \max_i B_i$ 
\State Let $X = 4B^2 + \frac{96}{5}\epsilon^2 +18B^2 \epsilon^{\frac{2}{p}} + 18B^{\frac{2}{p}} $
\State Set $\eta = \min\{(\epsilon + B)^{\frac{1}{p}}, 1 - \epsilon \}$
\State Let $A = \max\{16(p + 1)!, 2(p + 1)! \cdot [(p!)^{\frac{p + 1}{p}} + 8 \cdot (2p!)^{\frac{1}{p}}], 2(p + 1)! \cdot ((p!)^{\frac{p + 1}{p}} + 4(2p \cdot p!)^{\frac{1}{p}}) \}$
\State Set $M = \max \{(\frac{8A X^{\frac{p + 1}{2p}}}{\eta^{p + 1}})^{\frac{p}{p + 1}}, (\frac{8Ap \cdot X^{\frac{p + 1}{2p}}}{\eta^{\frac{p^2 - 1}{p}}})^{\frac{p}{p + 1}}, (\epsilon + B)^{\frac{-p - 2}{p}}, 8 L_p) \} $
\State Set $T = \lceil \frac{8A \Delta}{M \eta^{p + 1}} \rceil $
\State Set $x^{(0)} = x^{(1)} = 0$, $D^{(i)} = $ None for $i \in \{1, \ldots, p \}$
\vspace{0.25em}
\For{$t = 1$ to $T $}
\vspace{0.25em}
\State $D_t^{(i)} = $ \textbf{HO-RVR}($\epsilon$, $b$, $x^{(t)}$, $x^{(t - 1)}$, $D_{t - 1}^{(i)})$
\State Set the next point $x^{(t + 1)}$ as \begin{flalign*}
            &x^{(t + 1)} = \argmin_{y: ||y - x^{(t)}|| \le \eta} \sum_{i = 1}^p \frac{1}{i!} D_t^{(i)}[y - x^{(t)}]^i + \frac{M}{(p + 1)!} ||y - x^{(t)}||^{p + 1}
        \end{flalign*}

\EndFor

\renewcommand{\algorithmicrequire}{\textbf{Output:}} 
\State \Return $\hat{x}$ chosen uniformly at random from $\{x^{(t)} \}_{t = 2}^{T + 1}$
\end{algorithmic}
\end{algorithm}

\begin{theorem}
    \label{thm:rvr-upperbound}
    For any function $F \in \mathcal{F}_p(\Delta, L_{1: p})$, with biased and stochastic $p^{th}$ order oracles in $\mathcal{O}(F, \sigma_{1: p}, B_{1: p})$, with probability at least $\frac{5}{8}$, Algorithm \ref{alg:var-reduction-pth-order} returns a point $\hat{x}$ such that:
    \begin{itemize}
        \item If $B = \Theta(1)$, then $||\nabla F(\hat{x})|| \le O(\epsilon + B)$ with at most \begin{flalign*}
            O\left(\frac{\Delta (\max_i \sigma_i)^2(\epsilon + B)^{\frac{1}{p}} + (\max_i \sigma_i)^2}{\epsilon^2} + \frac{\Delta (\epsilon + B)^{\frac{1}{p}} + 1}{\epsilon}\right)
        \end{flalign*} queries to the stochastic and biased derivative oracles.
        \vspace{0.25em}
        \item If $B \ge \omega(1)$, then $||\nabla F(\hat{x})|| \le O((\epsilon^2 + B^2)^{\frac{1}{2}}(\epsilon + B))$ with at most \begin{flalign*}
            &O\left(\frac{(\max_i \sigma_i)^2}{\epsilon^2(\epsilon + B)^{\frac{p + 1}{p}}} + \frac{1}{\epsilon(\epsilon + B)^{\frac{p + 1}{p}}}\right)
        \end{flalign*} queries to the stochastic and biased derivative oracles. 
    \end{itemize}
\end{theorem} 
\begin{proof}[Proof Sketch]
    Here, we provide a brief proof sketch, deferring the full proof to Appendix \ref{appendix:rvr_upper_bound_proof}. In Lemma \ref{lemma:biasderivativediff}, we prove that \begin{flalign*}
        &\mathbb{E}||D^{(i)}(x^{(t)}) - \nabla ^{i} F(x^{(t)}) ||^2 \le 4B^2 + \frac{96}{5} \epsilon^2 + 18B^2 \epsilon^{\frac{2}{p}} + 18B^2
    \end{flalign*}  thereby establishing a bound on the difference in the derivative estimate versus the true derivative for all derivative orders. We then derive an upper bound for $\Pr(||\nabla F(\hat{x})|| \ge \frac{9M}{8p!} \eta^p)$ in terms of $B, \epsilon$, and the hyperparameters of Algorithm \ref{alg:var-reduction-pth-order} in Lemma \ref{lemma:rvrgradprobbound}. Plugging the parameters in gives an upper bound of $\frac{3}{8}$ for $\Pr(||\nabla F(\hat{x})|| \ge \frac{9M}{8p!} \eta^p)$, which finishes the proof.  
\end{proof}
 Notably, from Theorem \ref{thm:rvr-upperbound}, one can observe the following important conclusions: \begin{itemize}
    \item Unlike the purely minibatch-based approach for derivative estimation (utilized in Algorithm \ref{alg:novariancereduction}), appealing to higher order information does provide advantages in the event where $B = \Theta(1)$. 
    \item In the $B = \Theta(1)$ setting, we have a significantly improved oracle complexity of $O(\epsilon^{-2})$ compared to the $O(\epsilon^{-3})$ complexity from Algorithm \ref{alg:novariancereduction}.
    \item Consider the scenario where $B \ge \omega(1)$. Unlike the constant bias regime, appealing to higher order derivatives does not yield a better oracle complexity. Our intuition is to why this is the case, is if $B \ge \omega(1)$, the stationary guarantee is quite weak, implying that there may be little to no difference between appealing to higher order information and not doing so in terms of oracle complexity. Moreover, for the $p = 2$ case, we have an improved oracle complexity from the variance reduction based scheme in Algorithm \ref{alg:rvr_pth_order} compared to Algorithm \ref{alg:novariancereduction} for all settings of $\{ B_i \}$ such that $B \ge \omega(1)$. 
\end{itemize}

\section{Conclusion}
\label{section: conclusion}
This paper extends the settings of deterministic derivative oracles and stochastic but unbiased oracles to consider derivative oracles that are both stochastic and biased. We provide a matching first order lower bound to complement previous first-order upper bounds~\cite{OnConvergenceSGDBiased} for this stochastic and biased setting. We further extend this lower bound for algorithms that use second-order, or higher, derivative information for finding $O(\epsilon + B)$ stationary points. Then, to complement these lower bounds, we developed trust region based methods that under certain bias regimes essentially match the corresponding lower bound. We then improved upon these algorithms by incorporating a higher order variance reduction scheme, which improves the oracle complexity for other ranges of bias, and in some cases, reveals advantages of appealing to higher-order derivative information. 
\newline \newline 
Regarding opportunities for future work, our upper bound in Theorem \ref{thm:upperbounddeterministic} only matches the corresponding higher order lower bound in Theorem \ref{thm:higherorderlowerbound} for $O(\epsilon)$ bias, leaving open the possibility of stronger bounds. Moreover, it would be interesting to consider additional cases where appealing to higher-order derivative information would be beneficial for algorithms relying on biased and stochastic oracle access. 

\bibliographystyle{plain}
\bibliography{references}

\newpage 
\appendix

\section{Theorems \ref{thm:firstorderlowerbound}  and \ref{thm:higherorderlowerbound} Proofs}
\label{appendix: lower_bound_proofs}
\subsection{Preliminaries}
We first introduce some important notational conventions used throughout this section. Given a $p^{th}$ order tensor $T \in \mathbb{R}^{d \times \ldots, \times d}$, we define the support of $T$ as the following: \begin{flalign*}
    &\supp(T) = \{i\in [d]: T_i \ne 0 \}
\end{flalign*} where $T_i$ is the $(p-1)$ order subtensor denoted by $[T_i]_{j_1, \ldots, j_{p - 1}} = T_{i, j_1, \ldots, j_{p - 1}}$. For a tuple of tensors $\mathcal{T} = (T^{(1)}, T^{(2)}, \ldots)$, we define \begin{flalign*}
    &\supp(\mathcal{T}) = \bigcup_{i} \supp(T_i)
\end{flalign*} Moreover, given $x \in \mathbb{R}^d$, let \begin{flalign*}
    &\prog_\alpha (x) = \max \{i \ge 0, |x_i| > \alpha \}
\end{flalign*} which represents the highest index of $x$ whose entry is at least $\alpha$ from zero. Notice that for any $\alpha_1, \alpha_2 \in [0, 1)$ such that $\alpha_1 < \alpha_2$, we have that $\prog_{\alpha_2}(x) < \prog_{\alpha_1}(x)$. For a tensor $T$, we define $\prog(T) = \max \{\supp\{T \} \}$ which represents the highest index in $\supp \{T \}$, and naturally for a collection of tensors $\mathcal{T} = \{T^{(i)} \}$, we define $\prog(\mathcal{T}) = \max_i \prog(T^{(i)})$. 
\vspace{0.5em}
\begin{definition}
    \label{definition:zerorespectingalg}
    A biased and stochastic algorithm $A$ is zero-respecting if for any function $F$ and pth-order oracle $O_F^p$, the iterates $\{x^{(t)}\}$ satisfy \begin{flalign*}
        &\supp(x^{(t)}) \subseteq \bigcup_{i < t} \supp(O_F^p(x^{(i)}, z^{(i)} , b^{(i)}))
    \end{flalign*} for all $t \in \mathbb{N}$. 
\end{definition}
\vspace{0.5em}
\begin{definition}
    A collection of derivative estimators $\tilde{\nabla}^1F(x, z, b), \ldots, \tilde{\nabla}^pF(x, z, b)$ for a function $F$ form a probability-$\rho$ zero-chain if
    \begin{flalign*}
     &\Pr(\exists x | \prog(\tilde{\nabla}^1F(x, z, b), \ldots, \tilde{\nabla}^p F(x, z, b)) = \prog_{\frac{1}{4}}(x) + 1) \le \rho
    \end{flalign*}
    and \begin{flalign*}
        &\Pr(\exists x | \prog(\tilde{\nabla}^1F(x, z, b), \ldots, \tilde{\nabla}^p F(x, z, b)) = \prog_{\frac{1}{4}}(x) + i) = 0
    \end{flalign*} for all $i > 1$. 
\end{definition}
\vspace{0.5em}
\begin{lemma} 
\label{lemma:lowerboundprob}
Let $\tilde{\nabla}^1F(x, z, b), \ldots, \tilde{\nabla}^p F(x, z, b)$ be a collection of probability-$\rho$ zero-chain derivative  estimators for $F: \mathbb{R}^T \rightarrow \mathbb{R}$, and let $O_F^p(x, z, b) = (\tilde{\nabla}^qF(x, z, b))_{q \in \{1, \ldots, p\}}$. Let $\{x^{(t)}_{A[O_F]} \}$ be a sequence of queries produced by algorithm $A$ interacting with $O_F^p$. Then, with probability at least $1 - \delta$, \begin{flalign*}
        &\prog(x^{(t)}) < T
    \end{flalign*} for all \begin{flalign*}
        &t \le \frac{T - \log(1/\delta)}{2 \rho}
    \end{flalign*}
\end{lemma}

\begin{proof}
    Proved in Lemma 16 of \citep{SecondOrderOptNonconvexStochastic}
\end{proof}

\begin{definition}
    Let \begin{flalign*}
    &F_T(x) = -\Psi(1) \Phi(1) + \sum_{i = 2}^{T} [\Psi(-x_{i - 1}) \Phi(-x_i) - \Psi(x_{i - 1}) \Phi(x_i)
\end{flalign*} where \begin{flalign*}
    &\Psi(x) = \begin{cases} 
    0 & \text{if } x \le \frac{1}{2} \\
    \exp(1 - \frac{1}{(2x - 1)^2})   & \text{if } x > \frac{1}{2}
    \end{cases}, \hspace{0.5em}\Phi(x) = \sqrt{e} \int_{-\infty}^x e^{-\frac{1}{2}t^2}dt
\end{flalign*}
\end{definition}

\begin{lemma}
    \label{lemma:hardfunctionproperties}
    For $F_T$, the following properties hold: 
    \begin{itemize}
        \item $F_T(0) - \inf_{x} F_T(x) \le \Delta_0 T$, where $\Delta_0 = 12$ 
        \item For all $p \ge 1$, the $p^{th}$ order derivatives of $F_t$ are $\ell_p$-Lipschitz continuous, where $\ell_p \le \exp(\frac{5}{2} p \log p + cp)$ for some $c < \infty$
        \item For all $x \in \mathbb{R}^T$, $p \in \mathbb{N}$, and $1 \le i \le T$, we have that $||\nabla_i^p F_T(x)||_\mathrm{op} \le \ell_{p - 1}$
        \item For all $x \in \mathbb{R}^T$, $p \in \mathbb{N}$, $\prog(\nabla^q F_T(x)) \le \prog_{\frac{1}{2}}(x) + 1$
        \item For all $x \in \mathbb{R}^T$, if $\prog_1(x) < T$, then $||\nabla F_T(x)|| \ge |\nabla_{\prog_1(x) + 1} F_T(x)| > 1$
        \item For all $x, y \in \mathbb{R}^d$, there exists a constant $C \ge 0$ such that $||F(x) - F(y) || \le C \sqrt{T} ||x - y||$ 
    \end{itemize}
\end{lemma}

\begin{proof}
    The first five statements follow from Lemma 2, Lemma 3, and Observation 3 of \citep{LowerBoundsI} and section G.1.1. of \citep{SecondOrderOptNonconvexStochastic}. We now prove the last statement coordinate-wise by considering three separate cases for a coordinate $i$: $1 < i < T$, $i = 1$, and $i = T$. First, let $1 < i < T$. So, we have that $\partial_i F_T(x)$
        \begin{flalign*}
            &= \frac{\partial}{\partial x_i} (\Psi(-x_{i - 1}) \Phi(-x_i) - \Psi(x_{i - 1}) \Phi(x_i)) + \frac{\partial}{\partial x_i} (\Psi(-x_i) \Phi(-x_{i + 1}) - \Psi(x_i) \Phi(x_{i + 1})) \\ 
            &= -\Psi(-x_{i - 1}) \Phi'(-x_i) - \Psi(x_{i - 1})\Phi'(x_i)  - \Psi'(-x_i) \Phi(-x_{i + 1}) - \Psi'(x_i) \Phi(x_{i + 1}) 
        \end{flalign*} Observe that by construction of $\Psi$, we have that $|\Psi(x)| \le e$ and $|\Psi'(x)| \le 5$. Also, observe that \begin{flalign*}
            &\Phi(x) = \sqrt{e} \int_{-\infty}^x e^{-t^2/2} \le \sqrt{e} \int_{-\infty}^\infty e^{-t^2/2} = \sqrt{2\pi e}
        \end{flalign*} and that $|\Phi'(x)| \le \sqrt{e}$. Therefore, it holds that \begin{flalign*}
            &|\partial_i F_T(x)| \le e \sqrt{e} + e \sqrt{e} + 5 \sqrt{2\pi e} + 5 \sqrt{2 \pi e} \le 51
        \end{flalign*} One can derive similar constants for the $i = 1$ and $i = T$ case. Let $C$ be the maximum of 51 and these constants. We have that \begin{flalign*}
            &||\nabla F_T(x)||_2 = (\sum_{i = 1}^T |\partial_i F_T(x)|^2)^{\frac{1}{2}} \le (\sum_{i = 1}^T C^2)^{\frac{1}{2}} = C \sqrt{T}
        \end{flalign*} The statement follows by applying norm equivalence in finite dimensional spaces.

\end{proof}

\begin{definition}
    For all $q$, define the derivative estimators used to be \begin{flalign*}
    &[\tilde{\nabla}^qF_T(x, z)]_i = (1 + \boldsymbol{1}\{i > \prog_{\frac{1}{4}}(x) \}(\frac{z}{\rho} - 1)) \cdot (\nabla_i^q F_T(x) + b_i^q(x))
\end{flalign*} where $b^q$ is such that $b_i^q(x) = 0$ for all $i > \prog_{1/4}(x) + 1$, $||b_i^q(x)|| \le B_q$, and $z \sim $ Bernoulli$(\rho)$.
\end{definition}
\vspace{0.5em}
\begin{lemma}
    The derivative estimators $\tilde{\nabla}^q F_T$ form a probability-$\rho$ zero-chain and satisfy: \begin{flalign*}
        &\mathbb{E}[||\tilde{\nabla}^q F_T(x, z) - \nabla^qF_T(x)||^2] \le \frac{2 \ell_{q - 1}^2 (1 - \rho)}{\rho} + 2B_q^2
    \end{flalign*}
\end{lemma}

\begin{proof}
    First, we prove that these derivative estimators form a probability-$\rho$ chain. First, by the definition of $F_T$ and $b_i$, we can immediately conclude that $[\tilde{\nabla}^qF_T(x, z)]_i = 0$ for all $i > \prog_{\frac{1}{4}}(x) + 1$. Now, when $i = \prog_{\frac{1}{4}}(x) + 1$, we have that $[\tilde{\nabla}^qF_T(x, z)]_i = \frac{z}{p} \cdot (\nabla_i^q F_T(x) + b_i(x))$. if $z = 0$ (with probability $1 - \rho$), then we have that $[\tilde{\nabla}^q F_T(x, z)]_i = 0$. So, the first condition follows. Let $\overline{\nabla}^qF(x)$ be a stochastic but unbiased estimator of $\nabla^q F(x)$. We then have that \begin{flalign*}
        &\mathbb{E}[||\tilde{\nabla}^q F_T(x, z) - \nabla^qF_T(x)||^2] \\ 
        &\le 2\mathbb{E}[||\tilde{\nabla}^qF_T(x, z) - \overline{\nabla}^qF_T(x, z)||^2] + 2 \mathbb{E}[||\overline{\nabla}^q F_T(x, z) - \nabla^q F_T(x, z)||^2] \\  
        &\le 2 ||b_q(x)||^2 + \frac{2 \ell_{q - 1}^2(1 - \rho)}{\rho} \\ 
        &\le 2B_q^2 + \frac{2 \ell_{q - 1}^2(1 - \rho)}{\rho}
    \end{flalign*} which finishes the proof. 
\end{proof}

\subsection{Theorem \ref{thm:firstorderlowerbound} Proof}

\begin{remark}
    Our oracle model is the same as in \cite{OnConvergenceSGDBiased}, when setting $M = m = 0$, $\sigma = \sigma_1$, $\zeta = B_1$, and finding a point $x$ where $||\nabla F(x)|| = O((\epsilon + B_1^2)^{1/2})$ 
\end{remark}

\begin{proof}
    The setting of constants $M, m, \sigma, \zeta$ follows from definition 1, assumption 3, and assumption 4 in \cite{OnConvergenceSGDBiased}. The last point follows from the fact that in Theorem 4 of \cite{OnConvergenceSGDBiased}, the goal was to have iterates $\{x_t\}$ such that \begin{flalign*}
        &\frac{1}{T} \sum_{t = 0}^{T - 1} \mathbb{E}||\nabla f(x^{(t)})||^2 = O(\epsilon + B_1^2)
    \end{flalign*} which we can equivalently express as \begin{flalign*}
       \mathbb{E} ||\nabla F(\hat{x})|| = O((\epsilon + B_1^2)^{\frac{1}{2}})
    \end{flalign*} where $\hat{x}$ is drawn uniformly from $\{x_t \}$. 
\end{proof} 

\begin{theorem}
    (Theorem \ref{thm:firstorderlowerbound} restated). 
    When $p = 1$, there exists $F \in \mathcal{F}_1(\Delta, L_1)$ and $(O_F^1, P_z) \in \mathcal{O}_1(F, \sigma_1, B_1)$ such that for any first-order zero-respecting  algorithm (definition \ref{definition:zerorespectingalg}) 
    where $\epsilon < \frac{1}{4}$ and $B_1 \le O(1)$, the minimum number of queries to obtain a $(\epsilon + B_1^2)^{\frac{1}{2}}$ stationary point with constant probability is bounded below by \begin{flalign*}
        &\Omega \left(\frac{\Delta L_1}{\epsilon + B_1^2} + \frac{\Delta L_1 \sigma_1^2}{\epsilon^2 + B_1^4}\right)
    \end{flalign*}
\end{theorem}

\begin{proof}
    We let $F_T^* = \alpha F_T(\beta x)$ for some constants $\alpha, \beta$ which we set in this proof. With probability at least $\frac{3}{4}$, we have that $\prog(x^{(t)}_{A[O_F^p]}) < T$ for all $t \le \frac{T - 2}{2 \rho}$. Since $\prog_1(x) \le \prog(x)$, we have that \begin{flalign*}
    &\mathbb{E}||\nabla F_T^*(x^{(t)}_{A[O_F^p]})|| = \alpha \beta \mathbb{E}||\nabla F_T(x^{(t)}_{A[O_F^p]})|| \ge \frac{\alpha \beta}{2}
\end{flalign*} and that \begin{flalign*}
    &\mathbb{E}||\tilde{\nabla}^q F_T^*(x, z) - \nabla^q F_T^*(x, z)||^2 \le \alpha^2 \beta^{2q} \left(\frac{2 \ell_{q - 1}^2 (1 - \rho)}{\rho} + 2B_q^2\right)
\end{flalign*} for all $q$. Notice that by construction of $F_T^*$, we have that \begin{itemize}
    \item $F_T^*(0) - \inf_x F_T^*(x) = \alpha (F_T(0) - \inf_x F_T(\alpha x)) \le \alpha \Delta_0 T$ 
    \item $||\nabla^{2} F_T^*(x)|| = \alpha \beta^{q + 1} ||\nabla^{2}F_T(\beta x)|| \le \alpha \beta^{2} \ell_1$ 
    \item $||\nabla F_T^*(x)|| \ge \alpha \beta ||\nabla F_T(x)|| \ge \frac{\alpha \beta}{2}$
\end{itemize} We also have that \begin{flalign*}
    &\mathbb{E}||\tilde{\nabla}F_T^*(x, z) - \nabla F_T^*(x, z)||^2 \le \alpha^2 \beta^2 \left(\frac{2\ell_0^2 (1 - \rho)}{\rho} + 2B_1^2\right)
\end{flalign*} We set constants such that \begin{itemize}
    \item $\alpha \Delta_0 T \le \Delta$ 
    \item $\alpha \beta^2 \ell_1 \le L_1$ 
    \item $\frac{\alpha \beta}{2} \ge (\epsilon + B_1^2)^{\frac{1}{2}}$
    \item $\alpha^2 \beta^2 (\frac{2\ell_0^2 (1 - \rho)}{\rho} + 2B_1^2) \le 2\sigma_1^2 + 2B_1^2$ $\implies$ $\alpha^2 \beta^2 (\frac{\ell_0^2 (1 - \rho)}{\rho} + B_1^2) \le \sigma_1^2 + B_1^2$ 
\end{itemize} First, let $\alpha = 2(\epsilon + B_1^2)^{\frac{1}{2}}/\beta$. We then set \begin{flalign*}
    &\rho = \min \left\{\frac{\alpha^2 \beta^2 \ell_0^2}{\sigma_1^2 + B_1^2 - \alpha^2 \beta^2 B_1^2}, 1 \right\} = \min \left\{\frac{4(\epsilon + B_1^2) \ell_0^2}{\sigma_1^2 + B_1^2 - 4(\epsilon + B_1^2)B_1^2}, 1 \right\}
\end{flalign*} With this choice of $\rho$, it's easy to check that \begin{flalign*}
    &\alpha^2 \beta^2 \left(\frac{\ell_0^2(1 - \rho)}{\rho} + B_1^2 \right) \le \sigma_1^2 + B_1^2
\end{flalign*} To satisfy the Lipschitz condition, we set \begin{flalign*}
    &\beta = \frac{L_1}{2(\epsilon + B_1^2)^{\frac{1}{2}} \ell_1}
\end{flalign*} and we set \begin{flalign*}
    &T = \lfloor \frac{\Delta}{\alpha \Delta_0 } \rfloor = \lfloor \frac{\Delta \beta}{2 \Delta_0 (\epsilon + B_1^2)^{\frac{1}{2}}} \rfloor 
\end{flalign*} By Lemma \ref{lemma:lowerboundprob}, we have that \begin{flalign*}
    &\frac{T - 2}{2 \rho} = \frac{1}{2 \rho} (\lfloor \frac{\Delta \beta}{2 \Delta_0 (\epsilon + B_1^2)^{\frac{1}{2}}} - 2 \rfloor ) \\ 
    &\ge \frac{1}{2\rho} \cdot \frac{\Delta \beta}{4 \Delta_0 (\epsilon + B_1)^{\frac{1}{2}}} \\ 
    &\ge \Omega\left(\frac{\sigma_1^2 + B_1^2 - 4(\epsilon + B_1^2)B_1^2}{8(\epsilon + B_1^2) \ell_0^2} \cdot \frac{\Delta}{4 \Delta_0 (\epsilon + B_1^2)^{\frac{1}{2}}} \cdot \frac{L_1}{2(\epsilon + B_1^2)^{\frac{1}{2}} \ell_1}\right) \\ 
    &= \Omega\left(\frac{\Delta L_1(\sigma_1^2 + (1 - 4\epsilon) B_1^2 - 4B_1^4)}{64 \Delta_0 \ell_0^2 \ell_1 (\epsilon + B_1^2)^2}\right)
\end{flalign*} Now, since $\epsilon < \frac{1}{4}$, we can continue to lower bound this expression as follows: \begin{flalign*}
    &\Omega\left(\frac{\Delta L_1(\sigma_1^2 - 4B_1^4)}{64 \Delta_0 \ell_0^2 \ell_1 (\epsilon + B_1^2)^2}\right) \\ 
    &\ge \Omega\left(\frac{\Delta L_1 \sigma_1^2}{64 \Delta_0 \ell_0^2 \ell_1 (\epsilon + B_1^2)^2} - \frac{4 \Delta L_1}{64 \Delta_0 \ell_0^2 \ell_1}\right) \\ 
    &\ge \Omega \left(\frac{\Delta L_1 \sigma_1^2}{\Delta_0\ell_0^2 \ell_1(\epsilon + B_1^2)^2}\right) \\
    &= \Omega\left(\frac{\Delta L_1 \sigma_1^2}{\Delta_0\ell_0^2 \ell_1(\epsilon^2 + B_1^4)}\right)
\end{flalign*} Now, considering the case where the derivative oracles are biased but not stochastic (i.e $\sigma_1 = 0$), we have the following conditions: \begin{itemize}
    \item $\alpha \Delta_0 T \le \Delta$ 
    \item $\alpha \beta^2 \ell_1 \le L_1$ 
    \item $\frac{\alpha \beta}{2} \ge (\epsilon + B_1^2)^{\frac{1}{2}}$
    \item $\alpha^2 \beta^2 (\frac{2 \ell_0^2(1 - \rho)}{\rho} + 2B_1^2) \le 2B_1^2 \implies \alpha^2 \beta^2 (\frac{\ell_0^2(1 - \rho)}{\rho} + B_1^2) \le B_1^2$
\end{itemize} Again, we let $\alpha = 2(\epsilon + B_1^2)^{\frac{1}{2}}/\beta$. We then set \begin{flalign*}
    &\rho = \min \left\{ \frac{\alpha^2 \beta^2 \ell_0^2}{B_1^2 (1 - \alpha^2 \beta^2)}, 1 \right\} = \min \left\{\frac{4(\epsilon + B_1^2) \ell_0^2}{B_1^2 - 4B_1^2 (\epsilon + B_1^2)}, 1 \right\}
\end{flalign*} We again set \begin{flalign*}
    &\beta = \frac{L_1}{2(\epsilon + B_1^2)^{\frac{1}{2}}\ell_1}
\end{flalign*} and \begin{flalign*}
    &T = \lfloor \frac{\Delta}{\alpha \Delta_0} \rfloor = \lfloor \frac{\Delta \beta}{2(\epsilon + B_1^2)^{\frac{1}{2}}\Delta_0} \rfloor 
\end{flalign*} By Lemma \ref{lemma:lowerboundprob}, we have that \begin{flalign*}
    &\frac{T - 2}{2\rho} = \frac{1}{2 \rho} (\lfloor \frac{\Delta \beta}{2(\epsilon + B_1^2)^{\frac{1}{2}}\Delta_0} \rfloor - 2) \\ 
    &\ge \frac{1}{2\rho} \cdot \frac{\Delta \beta}{4 \Delta_0 (\epsilon + B_1^2)^{\frac{1}{2}}} \\ 
    &\ge \Omega\left(\frac{\Delta \beta}{\Delta_0(\epsilon + B_1^2)^{\frac{1}{2}}}\right)
\end{flalign*} since for all $B_1 \le O(1)$, $\rho = \Theta(1)$. When we further lower bound this expression, we have that \begin{flalign*}
    &\Omega \left(\frac{\Delta \beta}{\Delta_0(\epsilon + B_1^2)^{\frac{1}{2}}}\right) \ge \Omega\left(\frac{\Delta L_1}{\Delta_0(\epsilon + B_1^2) \ell_1}\right) = \Omega\left(\frac{\Delta L_1}{\epsilon + B_1^2}\right)
\end{flalign*} Combining the two lower bound expressions together (as in \cite{optmethodstoccompositeopt}) yields the matching lower bound: \begin{flalign*}
    &\Omega\left(\frac{\Delta L_1}{\epsilon + B_1^2} + \frac{\Delta L_1 \sigma_1^2}{\epsilon^2 + B_1^4}\right)
\end{flalign*}
\end{proof}

\subsection{Theorem \ref{thm:higherorderlowerbound} Proof}

\begin{theorem}
    (Theorem \ref{thm:higherorderlowerbound} restated). 
    For all $p \ge 2$, $\Delta$, $L_{1: p}$, $\sigma_{1: p} > 0$, $\epsilon < \sqrt{\sigma_1}$, and $B \le \frac{\sqrt{3}}{2} \sigma_1$, there exists $F \in \mathcal{F}_p(\Delta, L_{1: p})$ and $(O_F^p, P_z) \in \mathcal{O}_p(F, \sigma_{1: p}, B_{1: p})$ such that for any $p^{th}$ order zero-respecting algorithm, the number of queries to obtain a point an $\epsilon + \max_{i} B_i$ stationary point with constant probability is bounded below by 
    \begin{flalign*}
        &\Omega(1) \cdot \frac{(\sigma_1^2 - 4(\epsilon + B)^2 B_1^2) \Delta}{32 (\epsilon + B)^3 \ell_0^2 \Delta_0} \cdot \\ 
    &\min_{q' \in \{1, \ldots, p \},  q \in \{2, \ldots, p \} } \min \left \{\left(\frac{\ell_0^2(\sigma_q^2 + B_q^2)}{2 \ell_{q - 1}^2(\sigma_1^2 + B_1^2 - 4(\epsilon + B)^2 B_1^2)}\right)^{\frac{1}{2(q - 1)}}, \left(\frac{\sigma_q^2 + B_q^2}{8(\epsilon + B)^2 B_q^2}\right)^{\frac{1}{2(q - 1)}}, \left(\frac{L_{q'}}{2 (\epsilon + B) \ell_{q'}}\right)^{\frac{1}{q'}} \right\} 
    \end{flalign*}
\end{theorem}

\begin{proof}
    We perform a similar argument to 
    that for the proof of Theorem \ref{thm:firstorderlowerbound}, except now accounting for the higher order Lipschitz constraints. 
    We have that \begin{flalign*}
    &\mathbb{E}||\nabla F_T^*(x_{A[O_F^p]}^{(t)})|| = \alpha \beta ||\nabla F_T(x_{A[O_F^p]}^{(t)})|| \ge \frac{\alpha \beta}{2}
\end{flalign*} and that \begin{flalign*}
    &\mathbb{E}||\tilde{\nabla}^q F_T^*(x, z) - \nabla^q F_T^*(x, z)||^2 \le \alpha^2 \beta^{2q} (\frac{2 \ell_{q - 1}^2 (1 - \rho)}{\rho} + 2B_q^2)
\end{flalign*} We now set constants such that \begin{itemize}
    \item $\alpha \Delta_0 T \le \Delta$  
    \item $\alpha \beta^{q + 1} \ell_q \le L_q$ 
    \item $\frac{\alpha \beta}{2} \ge \epsilon + \max_{j} B_j$ 
    \item $\alpha^2 \beta^{2q} (\frac{2 \ell_{q - 1}^2 (1 - \rho)}{\rho} + 2B_q^2) \le 2 \sigma_q^2 + 2B_q^2 \implies \alpha^2 \beta^{2q} (\frac{ \ell_{q - 1}^2 (1 - \rho)}{\rho} + B_q^2) \le  \sigma_q^2 + B_q^2$
\end{itemize} First, let \begin{flalign*}
    &\alpha = \frac{2(\epsilon + \max_j B_j)}{\beta}
\end{flalign*} Now, we set \begin{flalign*}
    &\rho = \min \left \{ \frac{\alpha^2 \beta^2 \ell_0^2}{\sigma_1^2 + B_1^2 - \alpha^2 \beta^2 B_1^2 }
, 1 \right\}
\end{flalign*} This implies that (up to constant factors) 
\begin{flalign*}
    &\alpha^2 \beta^{2q} (\frac{\ell_{q - 1}^2 (1 - \rho)}{\rho} + B_q^2) \\ 
    &\le \alpha^2 \beta^{2q} (\frac{\ell_{q - 1}^2}{\rho} + B_q^2) \\
    &\le \alpha^2 \beta^{2q}(\frac{\ell_{q - 1}^2 (\sigma_1^2 + B_1^2 - \alpha^2 \beta^{2} B_1^2)}{\alpha^2 \beta^2 \ell_0^2} + B_q^2)  \\ 
    &\le \frac{\beta^{2(q - 1)}\ell_{q - 1}^2 (\sigma_1^2 + B_1^2 - \alpha^2 \beta^2 B_1^2)}{\ell_0^2} + \alpha^2 \beta^{2q} B_q^2 \\  
    &= \frac{\beta^{2(q - 1)}\ell_{q - 1}^2 (\sigma_1^2 + B_1^2 - 4 (\epsilon + B)^2 B_1^2)}{\ell_0^2} + \alpha^2 \beta^{2q} B_q^2 
\end{flalign*} where $B = \max_j B_j$. Letting \begin{flalign*}
    &\frac{\beta^{2(q - 1)}\ell_{q - 1}^2 (\sigma_1^2 + B_1^2 - 4 (\epsilon + B)^2 B_1^2)}{\ell_0^2} + \alpha^2 \beta^{2q} B_q^2 \le \sigma_q^2 + B_q^2
\end{flalign*} and solving for $\beta$ such that \begin{flalign*}
    &\frac{\beta^{2(q - 1)}\ell_{q - 1}^2 (\sigma_1^2 + B_1^2 - 4 (\epsilon + B)^2 B_1^2)}{\ell_0^2} \le \frac{\sigma_q^2 + B_q^2}{2}
\end{flalign*} and \begin{flalign*}
    &\alpha^2 \beta^{2q} B_q^2 \le \frac{\sigma_q^2 + B_q^2}{2}
\end{flalign*} and the $L_q$-condition holds, yields \begin{flalign*}
    &\beta = \min_{q' \in \{1, \ldots, p \},  q \in \{2, \ldots, p \} } \min \left\{\left(\frac{\ell_0^2(\sigma_q^2 + B_q^2)}{2 \ell_{q - 1}^2(\sigma_1^2 + B_1^2 - 4(\epsilon + B)^2 B_1^2)}\right)^{\frac{1}{2(q - 1)}}, \left(\frac{\sigma_q^2 + B_q^2}{8(\epsilon + B)^2 B_q^2}\right)^{\frac{1}{2(q - 1)}}, \left(\frac{L_{q'}}{2 (\epsilon + B) \ell_{q'}}\right)^{\frac{1}{q'}} \right\}
\end{flalign*} Setting \begin{flalign*}
    &T = \lfloor \frac{\Delta}{\alpha \Delta_0} \rfloor = \lfloor \frac{\Delta \beta}{2 \Delta_0 (\epsilon + B)} \rfloor 
\end{flalign*} We now have that (assuming $T \ge 5$) \begin{flalign*}
    &\frac{T - 2}{2 \rho} = \frac{1}{2 \rho}(\lfloor \frac{\Delta \beta}{2 \Delta_0 (\epsilon + B)} \rfloor - 2) \\ 
    &\ge \frac{1}{2 \rho} \cdot \frac{\Delta \beta}{4 \Delta_0 (\epsilon + B)} \\ 
    &\ge \frac{\sigma_1^2 + B_1^2 - \alpha^2 \beta^2 B_1^2}{2 \alpha^2 \beta^2 \ell_0^2} \cdot \frac{\Delta}{4 \Delta_0 (\epsilon + B)} \cdot \\ 
    &\min_{q' \in \{1, \ldots, p \},  q \in \{2, \ldots, p \} } \min \left \{\left(\frac{\ell_0^2(\sigma_q^2 + B_q^2)}{2 \ell_{q - 1}^2(\sigma_1^2 + B_1^2 - 4(\epsilon + B)^2 B_1^2)}\right)^{\frac{1}{2(q - 1)}}, \left(\frac{\sigma_q^2 + B_q^2}{8(\epsilon + B)^2 B_q^2}\right)^{\frac{1}{2(q - 1)}}, \left(\frac{L_{q'}}{2 (\epsilon + B) \ell_{q'}}\right)^{\frac{1}{q'}} \right\} \\ 
    &\ge \frac{(\sigma_1^2 - 4(\epsilon + B)^2 B_1^2) \Delta}{32 (\epsilon + B)^3 \ell_0^2 \Delta_0} \cdot \\ 
    &\min_{q' \in \{1, \ldots, p \},  q \in \{2, \ldots, p \} } \min \left \{\left(\frac{\ell_0^2(\sigma_q^2 + B_q^2)}{2 \ell_{q - 1}^2(\sigma_1^2 + B_1^2 - 4(\epsilon + B)^2 B_1^2)}\right)^{\frac{1}{2(q - 1)}}, \left(\frac{\sigma_q^2 + B_q^2}{8(\epsilon + B)^2 B_q^2}\right)^{\frac{1}{2(q - 1)}}, \left(\frac{L_{q'}}{2 (\epsilon + B) \ell_{q'}}\right)^{\frac{1}{q'}} \right \}  
\end{flalign*} which finishes the proof. 
 
\end{proof}
\newpage 
\section{Theorem \ref{thm:upperbounddeterministic} Proof}
\label{appendix_a:minibatch_deriv_proof}
\input{appendix_a_proofs_higher_order}

\newpage 
\section{Theorem \ref{thm:rvr-upperbound} Proof}
\label{appendix:rvr_upper_bound_proof}

\begin{lemma}
    \label{lemma:biasderivativediff}
    Let $F \in \mathcal{F}_p(\Delta, L_{1: p})$. For any biased and stochastic oracle in $\mathcal{O}_p(F, \sigma_{1: p}, B_{1: p})$, let $\{D^i(x^{(t)})\}$ represent the sequence of $i^{th}$ derivative iterates generated by Algorithm \ref{alg:rvr_pth_order}. Let $B = \max_{1 \le i \le p} B_i$. Then, we have that \begin{flalign*}
        &\mathbb{E}[||D^i(x^{(t)}) - \nabla^i F(x^{(t)})||^2] \le 4B^2 + \frac{96}{5} \epsilon^2 + 18B^2 \epsilon^{2/p} + 18B^{2/p}
    \end{flalign*} for all $1 \leq i \leq p$ and all $t \ge 1$.
\end{lemma}

\begin{proof}
    We can first say that \begin{flalign*}
        &\mathbb{E}||D^i(x^{(1)}) - \nabla^i F(x^{(1)})||^2 \\
        &=\mathbb{E}[||b_i(x^{(1)}) + \frac{1}{n_1} \sum_{j = 1}^{n_1} \epsilon_1(x^{(1)}, z^{(1, j)})||^2] \\ 
        &\le 2B_i^2 + \frac{2}{n_i^2} \sum_{j = 1}^{n_i} ||\epsilon_i(x^{(1)}, z^{(1, j)})||^2 \le 2B_i^2 + \frac{2\sigma_i^2}{n_i} \le 2B_i^2 + \frac{2 \epsilon^2}{5}
    \end{flalign*} Let $e^{(t)} = D^i_t(x^{(t)})- \nabla F^i(x^{(t)})$, and we have that \begin{flalign*}
        &\mathbb{E}[||e^{(t)}||^2 |b^{(t)}] = b^{(t)} \cdot \mathbb{E}[||e^{(t)}||^2 |C^{(t)} = 1] + (1 - b^{(t)}) \cdot \mathbb{E}[||e^{(t)}||^2 |C^{(t)} = 0] 
    \end{flalign*} where \begin{flalign*}
        &\mathbb{E}[||e^{(t)}||^2 |C^{(t)} = 1] \le 2B_i^2 + \frac{2 \sigma_i^2}{n_i} \le 2B_i^2 + \frac{2 \epsilon^2}{5}
    \end{flalign*} We now say that \begin{flalign*}
        &\mathbb{E}[||e^{(t)}||^2 |C^{(t)} = 0] \\
        &\le \mathbb{E}[||e^{(t - 1)} + \mathbb{E}[\psi^{(t)} |\mathcal{G}^{(t)}]||^2]  + \mathbb{E}[||\psi^{(t)} - \mathbb{E}[\psi^{(t)}| \mathcal{G}^{(t)}]||^2] \\ 
        &\le \mathbb{E}[(1 + \frac{2}{b^{(t)}}) \cdot ||e^{(t - 1)}||^2] + \mathbb{E}[(1 + \frac{2}{b^{(t)}}) \cdot ||\mathbb{E}[\psi^{(t)}|\mathcal{G}^{(t)}]||^2] + \mathbb{E}[||\psi^{(t)} - \mathbb{E}[\psi^{(t)}| \mathcal{G}^{(t)}]||^2] 
    \end{flalign*} where the first step follows from the fact that $\mathcal{G}^{(t)}$ is a measurable set, and the second step is by Young's inequality. Above, we have that \begin{flalign*}
        &\psi^{(t)} = e^{(t)} - e^{(t - 1)} = \sum_{k = 1}^{K^{(t)}} \tilde{\nabla}^{i + 1}F(x^{(t, k - 1)}, z^{(t, k)}, b_i)(x^{(t, k)} - x^{(t, k - 1)}) - \nabla^iF(x^{(t)}) + \nabla^i F(x^{(t - 1)})
    \end{flalign*} We can calculate that \begin{flalign*}
        &\mathbb{E}[\psi^{(t)} |\mathcal{G}^{(t)}] = \sum_{k = 1}^{K^{(t)}} (\nabla^{i + 1}F(x^{(t, k - 1)}) + b_{i + 1}(x^{(t, k - 1)}))(x^{(t, k)} - x^{(t, k - 1)}) - \nabla^i F(x^{(t)}) + \nabla^i F(x^{(t - 1)})
    \end{flalign*} which implies that \begin{flalign*} 
        &||\mathbb{E}[\psi^{(t)} |\mathcal{G}^{(t)}]|| \\
        &\le \sum_{k = 1}^{K^{(t)}} ||(\nabla^i F(x^{(t, k)}) - \nabla^iF(x^{(t, k - 1)}) - \nabla^{i + 1} F(x^{(t, k - 1)})(x^{(t, k)} - x^{(t, k - 1)})|| \\ 
        &+ \sum_{k = 1}^{K^{(t)}} ||b_{i + 1}(x^{(t, k - 1)})|| \cdot ||x^{(t, k)} - x^{(t, k - 1)}|| \\ 
        &\le K^{(t)} \cdot \frac{L_{i + 1}}{2} \cdot (\frac{||x^{(t)} - x^{(t - 1)}||}{K^{(t)}})^2 + B_{i + 1} ||x^{(t)} - x^{(t - 1)}|| \\ 
        &\le \frac{b^{(t)} \epsilon}{10} + B_{i + 1} \eta \\ 
        &\le \frac{b^{(t)} \epsilon}{10} + B_{i + 1} b^{(t)}(\epsilon + \max_i B_i)^{1/p}
    \end{flalign*} We can also derive that \begin{flalign*}
        &\mathbb{E}[||\psi^{(t)} - \mathbb{E}[\psi^{(t)} | \mathcal{G}^{(t)}]||^2] \\ 
        &= \frac{1}{(K^{(t)})^2} \sum_{k = 1}^{K^{(t)}} \mathbb{E}[||(\tilde{\nabla}^{i + 1}F(x^{(t, k - 1)}, z^{(t, k)}) - \nabla^{i + 1} F(x^{(t, k- 1)}) - b_{i + 1}(x^{(t, k - 1)})(x^{(t)} - x^{(t - 1)})||^2|\mathcal{G}^{(t)}] \\ 
        &\le \frac{1}{(K^{(t)})^2} \sum_{k = 1}^{K^{(t)}} \mathbb{E}[||(\tilde{\nabla}^{i + 1}F(x^{(t, k - 1)}, z^{(t, k)}) - \nabla^{i + 1} F(x^{(t, k- 1)}) - b_{i + 1}(x^{(t, k - 1)})||^2_\mathrm{op} | \mathcal{G}^{(t)}] \cdot ||x^{(t)} - x^{(t - 1)}||^2 \\
        &\le \sigma_{i + 1}^2 \frac{||x^{(t)} - x^{(t - 1)}||^2}{K^{(t)}} \le b^{(t)} \frac{\epsilon^2}{5}
    \end{flalign*} Combining both of these inequalities together, we have that \begin{flalign*}
        &\mathbb{E}||e^{(t)}||^2 \\  
        &= b^{(t)} (2B_i^2 + \frac{2 \epsilon^2}{5}) + (1 - b^{(t)}) \cdot \mathbb{E}[||e^{(t)}||^2 | C^{(t)} = 0] \\ 
        &\le b^{(t)} (2B_i^2 + \frac{2 \epsilon^2}{5}) + \mathbb{E}[(1 - b^{(t)})(1 + \frac{2}{b^{(t)}}) ||e^{(t - 1)}||^2 + (1 - b^{(t)})(1 + \frac{2}{b^{(t)}})(\frac{b^{(t)} \epsilon}{10} + B_{i + 1} b^{(t)}(\epsilon + \max_i B_i)^{\frac{1}{p}})^2] \\ 
        &+ \mathbb{E}[(1 - b^{(t)})\cdot \frac{b^{(t)} \epsilon^2}{5}] \\ 
        &\le \mathbb{E}[b^{(t)} \cdot (2B^2 + \frac{2 \epsilon^2}{5})] + (1 - \frac{\mathbb{E}[b^{(t)}]}{2})||e^{(t - 1)}||^2 + \mathbb{E}[3b^{(t)} \cdot (\epsilon + B(\epsilon + B)^{1/p})^2 + b^{(t)} \frac{\epsilon^2}{5}] \\ 
        &\le (1 - \frac{\mathbb{E}[b^{(t)}]}{2})||e^{(t - 1)}||^2 + \mathbb{E}[b^{(t)}](2B^2 + \frac{3 \epsilon^2}{5} + 3(\epsilon + B(\epsilon + B)^{1/p})^2) \\ 
        &\le (1 - \frac{\mathbb{E}[b^{(t)}]}{2}) ||e^{(t - 1)}||^2 + \mathbb{E}[b^{(t)}] \cdot (2B^2 + \frac{3 \epsilon^2}{5} + 3(\epsilon + B \epsilon^{1/p} + B^{1/p})^2) \\ 
        &\le (1 - \frac{\mathbb{E}[b^{(t)}]}{2}) ||e^{(t - 1)}||^2 + \mathbb{E}[b^{(t)}] \cdot (2B^2 + \frac{3 \epsilon^2}{5} + 9 \epsilon^2 + 9 B^2 \epsilon^{2/p} + 9 B^{2/p}) \\ 
        &= (1 - \frac{\mathbb{E}[b^{(t)}]}{2}) ||e^{(t - 1)}||^2 + \mathbb{E}[b^{(t)}] \cdot (2B^2 + \frac{48 \epsilon^2}{5} + 9 B^2 \epsilon^{2/p} + 9 B^{2/p}) \\ 
        &= (1 - \frac{\mathbb{E}[b^{(t)}]}{2}) ||e^{(t - 1)}||^2 + \frac{\mathbb{E}[b^{(t)}]}{2} \cdot (4B^2 + \frac{96 \epsilon^2}{5} + 18 B^2 \epsilon^{2/p} + 18 B^{2/p})
    \end{flalign*} which implies that \begin{flalign*}
        &\mathbb{E}||e^{(t)}||^2 \\ 
        &\le (4B^2 + \frac{96}{5} \epsilon^2 + 18B^2 \epsilon^{2/p} + 18B^{2/p}) - (2B^2 + \frac{94}{5} \epsilon^2 + 18B^2 \epsilon^{2/p} + 18B^{2/p}) \prod_{s = 2}^t (1 - \frac{b^{(s)}}{2} ) \\ 
        &\le 4B^2 + \frac{96}{5} \epsilon^2 + 18B^2 \epsilon^{2/p} + 18B^{2/p}
    \end{flalign*} which finishes the proof. 
\end{proof}

\begin{lemma}
    \label{lemma:rvrgradprobbound}
    It holds that \begin{flalign*}
        &\Pr(||\nabla F(\hat{x})|| \ge \frac{9M}{8p!} \eta^p) \le \frac{16(p + 1)!}{M \eta^{p + 1}T}\Delta + \frac{16(p + 1)!}{M \eta^{p + 1}} \cdot [\frac{(p!)^{\frac{p + 1}{p}}}{8 \sqrt[p]{M}} + (\frac{2p!}{M})^{\frac{1}{p }}] \cdot (4B^2 + \frac{96}{5} \epsilon^2 + 18B^2 \epsilon^{2/p} + 18B^{2/p})^{\frac{p + 1}{2p}} \\ 
        &+ \frac{16p(p + 1)!}{M \eta^{\frac{p^2 - 1}{p}}} \cdot [\frac{(p!)^{\frac{p + 1}{p}}}{8 \sqrt[p]{M}} + \frac{1}{2} (\frac{2p \cdot p!}{M})^{\frac{1}{p }}] \cdot (4B^2 + \frac{96}{5} \epsilon^2 + 18B^2 \epsilon^{2/p} + 18B^{2/p})^{\frac{p + 1}{2p}} 
    \end{flalign*}
\end{lemma}

\begin{proof}
    From Lemma \ref{lemma:randomvarderivatives}, we have that \begin{flalign*}
        &\mathbb{E}[F(x^{(t)}) - F(x^{(t + 1)})] \\ 
        &\ge \frac{M \eta^{p + 1}}{16(p + 1)!} \Pr(||\nabla F(x^{(t + 1)})|| \ge \frac{9M}{8p!} \eta^p) - [\frac{(p!)^{\frac{p + 1}{p}}}{8 \sqrt[p]{M}} + (\frac{2p!}{M})^{\frac{1}{p }}] \cdot \mathbb{E}[||\nabla F(x^{(t)}) - D^{(1)}(x^{(t)})||^{\frac{p + 1}{p}}] \\ 
        &- \eta^{\frac{p + 1}{p}}(\frac{(p!)^{\frac{p + 1}{p}}}{8 \sqrt[p]{M}} + \frac{1}{2} (\frac{2p \cdot p!}{M})^{\frac{1}{p }}) \sum_{i = 2}^p \mathbb{E}[||\nabla^i F(x^{(t)}) - D^i(x^{(t)})||^{\frac{p + 1}{p}}] \\ 
        &\ge \frac{M \eta^{p + 1}}{16(p + 1)!} \Pr(||\nabla F(x^{(t + 1)})|| \ge \frac{9M}{8p!} \eta^p) - [\frac{(p!)^{\frac{p + 1}{p}}}{8 \sqrt[p]{M}} + (\frac{2p!}{M})^{\frac{1}{p }}] \cdot (\mathbb{E}[||\nabla F(x^{(t)}) - D^{(1)}(x^{(t)})||^2])^{\frac{p + 1}{2p}} \\ 
        &- \eta^{\frac{p + 1}{p}} \cdot [\frac{(p!)^{\frac{p + 1}{p}}}{8 \sqrt[p]{M}} + \frac{1}{2} (\frac{2p \cdot p!}{M})^{\frac{1}{p }}] \cdot \sum_{i = 2}^p (\mathbb{E}[||\nabla^{(i)} F(x^{(t)}) - D^{(i)}(x^{(t)})||^2])^{\frac{p + 1}{2p}} \\ 
        &\ge \frac{M \eta^{p + 1}}{16(p + 1)!} \Pr(||\nabla F(x^{(t + 1)})|| \ge \frac{9M}{8p!} \eta^p) - [\frac{(p!)^{\frac{p + 1}{p}}}{8 \sqrt[p]{M}} + (\frac{2p!}{M})^{\frac{1}{p }}] \cdot (4B^2 + \frac{96}{5} \epsilon^2 + 18B^2 \epsilon^{\frac{2}{p}} + 18B^{\frac{2}{p}})^{\frac{p + 1}{2p}} \\ 
        &- p\eta^{\frac{p + 1}{p}} \cdot [\frac{(p!)^{\frac{p + 1}{p}}}{8 \sqrt[p]{M}} + \frac{1}{2} (\frac{2p \cdot p!}{M})^{\frac{1}{p }}] \cdot (4B^2 + \frac{96}{5} \epsilon^2 + 18B^2 \epsilon^{\frac{2}{p}} + 18B^{\frac{2}{p}})^{\frac{p + 1}{2p}}
    \end{flalign*} Telescoping this recurrence from $t = 1$ to $T$ gives \begin{flalign*}
        &\mathbb{E}[F(x^{(1)}) - F(x^{(T + 1)})] \\
        &\ge \frac{M \eta^{p + 1}T}{16(p + 1)!} \Pr(||\nabla F(\hat{x})|| \ge \frac{9M}{8p!} \eta^p) - T \cdot [\frac{(p!)^{\frac{p + 1}{p}}}{8 \sqrt[p]{M}} + (\frac{2p!}{M})^{\frac{1}{p }}] \cdot (4B^2 + \frac{96}{5} \epsilon^2 + 18B^2 \epsilon^{2/p} + 18B^{2/p})^{\frac{p + 1}{2p}} \\ 
        &- p\eta^{\frac{p + 1}{p}} \cdot [\frac{(p!)^{\frac{p + 1}{p}}}{8 \sqrt[p]{M}} + \frac{1}{2} (\frac{2p \cdot p!}{M})^{\frac{1}{p }}] \cdot (4B^2 + \frac{96}{5} \epsilon^2 + 18B^2 \epsilon^{2/p} + 18B^{2/p})^{\frac{p + 1}{2p}}
    \end{flalign*} which implies that \begin{flalign*}
        &\Pr(||\nabla F(\hat{x})|| \ge \frac{9M}{8p!} \eta^p) \le \frac{16(p + 1)!}{M \eta^{p + 1}T}\Delta + \frac{16(p + 1)!}{M \eta^{p + 1}} \cdot [\frac{(p!)^{\frac{p + 1}{p}}}{8 \sqrt[p]{M}} + (\frac{2p!}{M})^{\frac{1}{p }}] \cdot (4B^2 + \frac{96}{5} \epsilon^2 + 18B^2 \epsilon^{2/p} + 18B^{2/p})^{\frac{p + 1}{2p}} \\ 
        &+ \frac{16p(p + 1)!}{M \eta^{\frac{p^2 - 1}{p}}} \cdot [\frac{(p!)^{\frac{p + 1}{p}}}{8 \sqrt[p]{M}} + \frac{1}{2} (\frac{2p \cdot p!}{M})^{\frac{1}{p }}] \cdot (4B^2 + \frac{96}{5} \epsilon^2 + 18B^2 \epsilon^{2/p} + 18B^{2/p})^{\frac{p + 1}{2p}}
    \end{flalign*} which finishes the proof. 
\end{proof}

\begin{theorem}
    Theorem (\ref{thm:rvr-upperbound} restated).
    For any function $F \in \mathcal{F}_p(\Delta, L_{1: p})$, with biased and stochastic $p^{th}$ order oracles in $\mathcal{O}(F, \sigma_{1: p}, B_{1: p})$, with probability at least $\frac{5}{8}$, Algorithm \ref{alg:var-reduction-pth-order} returns a point $\hat{x}$ such that:
    \begin{itemize}
        \item If $\max_i B_i = \Theta(1)$, then $||\nabla F(\hat{x})|| \le O(\epsilon + \max_i B_i)$ with at most \begin{flalign*}
            O(\frac{\Delta (\max_i \sigma_i)^2(\epsilon + B)^{\frac{1}{p}} + (\max_i \sigma_i)^2}{\epsilon^2} + \frac{\Delta (\epsilon + B)^{\frac{1}{p}} + 1}{\epsilon})
        \end{flalign*} queries to the stochastic and biased derivative oracles.
        \vspace{0.25em}
        \item If $\max_i B_i \ge \omega(1)$, then $||\nabla F(\hat{x})|| \le O((\epsilon^2 + B^2)^{\frac{1}{2}}(\epsilon + B))$ with at most \begin{flalign*}
            &O(\frac{(\max_i \sigma_i)^2}{\epsilon^2(\epsilon + B)^{\frac{p + 1}{p}}} + \frac{1}{\epsilon(\epsilon + B)^{\frac{p + 1}{p}}})
        \end{flalign*} queries to the stochastic and biased derivative oracles. 
    \end{itemize}
where $B = \max_i B_i$. 
\end{theorem} 

\begin{proof}
    From Lemma \ref{lemma:rvrgradprobbound}, we have that \begin{flalign*}
        &\Pr(||\nabla F(\hat{x})|| \ge \frac{9M}{8p!} \eta^p) \\ 
        &\le \frac{16(p + 1)!}{M \eta^{p + 1}T}\Delta + \frac{16(p + 1)!}{M \eta^{p + 1}} \cdot [\frac{(p!)^{\frac{p + 1}{p}}}{8 \sqrt[p]{M}} + (\frac{2p!}{M})^{\frac{1}{p }}] \cdot (4B^2 + \frac{96}{5} \epsilon^2 + 18B^2 \epsilon^{2/p} + 18B^{2/p})^{\frac{p + 1}{2p}} \\ 
        &+ \frac{16p(p + 1)!}{M \eta^{\frac{p^2 - 1}{p}}} \cdot [\frac{(p!)^{\frac{p + 1}{p}}}{8 \sqrt[p]{M}} + \frac{1}{2} (\frac{2p \cdot p!}{M})^{\frac{1}{p }}] \cdot (4B^2 + \frac{96}{5} \epsilon^2 + 18B^2 \epsilon^{2/p} + 18B^{2/p})^{\frac{p + 1}{2p}}  \\ 
        &= \frac{16(p + 1)!}{M \eta^{p + 1}T}\Delta + \frac{16(p + 1)!}{M \eta^{p + 1}}[\frac{(p!)^{\frac{p + 1}{p}} + 8 \cdot (2p!)^{\frac{1}{p}}}{8M^{\frac{1}{p}}}] \cdot (4B^2 + \frac{96}{5} \epsilon^2 + 18B^2 \epsilon^{2/p} + 18B^{2/p})^{\frac{p + 1}{2p}} \\ 
        &+ \frac{16p(p + 1)!}{M \eta^{\frac{p^2 - 1}{p}}} \cdot [\frac{(p!)^{\frac{p + 1}{p}} + 4(2p \cdot p!)^{\frac{1}{p}}}{8M^{\frac{1}{p}}}] \cdot (4B^2 + \frac{96}{5} \epsilon^2 + 18B^2 \epsilon^{2/p} + 18B^{2/p})^{\frac{p + 1}{2p}} \\ 
        &= \frac{16(p + 1)!}{M \eta^{p + 1}T}\Delta + \frac{2(p + 1)!}{M^{\frac{p + 1}{p}} \eta^{p + 1}}[(p!)^{\frac{p + 1}{p}} + 8 \cdot (2p!)^{\frac{1}{p}}] \cdot (4B^2 + \frac{96}{5} \epsilon^2 + 18B^2 \epsilon^{2/p} + 18B^{2/p})^{\frac{p + 1}{2p}} \\ 
        &+ \frac{2p(p + 1)!}{M^{\frac{p + 1}{p}}\eta^{\frac{p^2 - 1}{p}}} \cdot [(p!)^{\frac{p + 1}{p}} + 4(2p \cdot p!)^{\frac{1}{p}}] \cdot (4B^2 + \frac{96}{5} \epsilon^2 + 18B^2 \epsilon^{2/p} + 18B^{2/p})^{\frac{p + 1}{2p}}
    \end{flalign*} Letting \begin{flalign*}
        &A = \max(16(p + 1)!, 2(p + 1)! \cdot [(p!)^{\frac{p + 1}{p}} + 8 \cdot (2p!)^{\frac{1}{p}}], 2(p + 1)! \cdot [(p!)^{\frac{p + 1}{p}} + 4(2p \cdot p!)^{\frac{1}{p}}] )
    \end{flalign*} 
    we get an upper bound of \begin{flalign*}
        &\frac{A \Delta}{M \eta^{p + 1} T} + \frac{A}{M^{\frac{p + 1}{p}} \eta^{p + 1}} \cdot (4B^2 + \frac{96}{5} \epsilon^2 + 18B^2 \epsilon^{2/p} + 18B^{2/p})^{\frac{p + 1}{2p}} \\ 
        &+ \frac{A}{M^{\frac{p + 1}{p}} \eta^{\frac{p^2 - 1}{p}}} \cdot (4B^2 + \frac{96}{5} \epsilon^2 + 18B^2 \epsilon^{2/p} + 18B^{2/p})^{\frac{p + 1}{2p}}
    \end{flalign*}
    Let $X = 4B^2 + \frac{96}{5} \epsilon^2 + 18 B^2 \epsilon^{2/p} + 18B^{2/p}$. Since we set \begin{flalign*}
        &M = \max \{(\frac{8A X^{\frac{p + 1}{2p}}}{\eta^{p + 1}})^{\frac{p}{p + 1}}, (\frac{8Ap \cdot X^{\frac{p + 1}{2p}}}{\eta^{\frac{p^2 - 1}{p}}})^{\frac{p}{p + 1}}, (\epsilon + B)^{\frac{-p - 2}{p}}, 8 L_p)  
    \end{flalign*} we have that \begin{flalign*}
        &\frac{A}{M^{\frac{p + 1}{p}} \eta^{p + 1}} (4B^2 + \frac{96}{5} \epsilon^2 + 18B^2 \epsilon^{2/p} + 18B^{2/p})^{\frac{p + 1}{2p}} \le \frac{1}{8}
    \end{flalign*} and \begin{flalign*}
        &\frac{pA}{M^{\frac{p + 1}{p}}\eta^{\frac{p^2 - 1}{p^2}}}(4B^2 + \frac{96}{5} \epsilon^2 + 18B^2 \epsilon^{2/p} + 18B^{2/p})^{\frac{p + 1}{2p}} \le \frac{1}{8}
    \end{flalign*} and through setting \begin{flalign*}
        &T = \lceil \frac{8 A \Delta}{M \eta^{p + 1}} \rceil > \frac{8A \Delta}{M \eta^{p + 1}}
    \end{flalign*} we have that \begin{flalign*}
        &\frac{A \Delta}{M \eta^{p + 1}T} < \frac{1}{8}
    \end{flalign*} Notice that if $B = \Theta(1)$, then $X = \Theta(1)$, implying that $M \le O(1)$, meaning that we have that \begin{flalign*}
        &\frac{5}{8} \le \Pr(||\nabla F(\hat{x})|| \le \frac{9M}{8p!} \eta^p)  \le \Pr(||\nabla F(\hat{x})|| \le O(1) \cdot (\epsilon + B))
    \end{flalign*} On the other hand, if $B \ge \omega(1)$, then $X = O(B^2 + \epsilon^2)$, which implies that 
    \newline 
    $M = \max(O(\frac{(B^2 + \epsilon^2)^{1/2}}{\eta^p}), O(\frac{(B^2 + \epsilon^2)^{\frac{1}{2}}}{\eta^{p - 1}}), O(1), O(1))$. Since $B \ge \omega(1)$, $1 - \epsilon < (\epsilon + B)^{\frac{1}{p}}$, so $\eta = 1 - \epsilon$. 
    Therefore, $M = O((B^2 + \epsilon^2)^{\frac{1}{2}})$. We then have that \begin{flalign*}
        &\frac{5}{8} \le \Pr(||\nabla F(\hat{x})|| \le \frac{9M}{8p!} \eta^p) \le \Pr(||\nabla F(\hat{x})|| \le O((\epsilon^2 + B^2)^{\frac{1}{2}}(\epsilon + B))
    \end{flalign*} 
    Let $M_i$ be the number of oracle queries for derivative order $i$, and let $M = \sum_{i = 1}^{p + 1} M_i$ be the total number of oracle queries.
    With regards to the oracle complexity, we have that \begin{flalign*}
        &\mathbb{E}[M] \\ 
        &\le \sum_{i = 1}^{p + 1} \mathbb{E}[M_i] \\ 
        &= T \sum_{i = 1}^{p + 1} \Pr(C = 1) \mathbb{E}[m_i | C = 1] + \Pr(C = 0) \mathbb{E}[m_i |C = 0] \\ 
        &= T \sum_{i = 1}^{p + 1} bn_i + (1 - b) K_i \\ 
        &\le T \sum_{i = 1}^{p + 1} b(\frac{5 \sigma_i^2}{\epsilon^2} + 1) + (1 - b)(\frac{5(\sigma_{i + 1}^2 + L_{i + 1} \epsilon)}{b \epsilon^2} + 1)
    \end{flalign*} Letting $\sigma = \max_i \sigma_i$, we upper bound the expression above by \begin{flalign*}
        &T \sum_{i = 1}^{p + 1} b(\frac{5 \sigma^2}{\epsilon^2} + 1) + (1 - b) (\frac{5(\sigma^2 + L_{i + 1} \epsilon)}{b \epsilon^2} + 1) \\ 
        &\le T \sum_{i = 1}^{p + 1} \frac{5b^2 \sigma^2 + 5 \sigma^2 + 5 L_{i + 1} \epsilon}{b \epsilon^2} + 2 \\ 
        &\le T \cdot O(\frac{\sigma^2}{\epsilon^2} + \frac{1}{\epsilon}) 
    \end{flalign*} We again analyze the following cases: $B = \Theta(1)$ and $B \ge \omega(1)$. If $B = \Theta(1)$, then \begin{flalign*}
        &\frac{8A \Delta}{M \eta^{p + 1}} \le \frac{8A \Delta}{(\epsilon + B)^{-\frac{1}{p}} \eta^{p + 1}} \le \frac{8A \Delta}{(\epsilon + B)^{\frac{-p - 2}{p}}(\epsilon + B)^{\frac{p + 1}{p}} \cdot O(1)} = O(1) \cdot 8A\Delta \cdot (\epsilon + B)^{\frac{1}{p}}
    \end{flalign*} Therefore, \begin{flalign*}
        &T \cdot O(\frac{\sigma^2}{\epsilon^2} + \frac{1}{\epsilon}) \le O(\frac{\Delta \sigma^2 (\epsilon + B)^{\frac{1}{p}} + \sigma^2}{\epsilon^2} + \frac{\Delta (\epsilon + B)^{\frac{1}{p}} + 1}{\epsilon})
    \end{flalign*} If $B \ge \omega(1)$, then \begin{flalign*}
        &\frac{8A \Delta}{M \eta^{p + 1}} = \frac{8A \Delta}{M (\epsilon + B)^{\frac{p + 1}{p}}} \le \frac{8A \Delta}{8L_p \cdot (\epsilon + B)^{\frac{p + 1}{p}}} \le O(\frac{\Delta}{L_p \cdot (\epsilon + B)^{\frac{p + 1}{p}}})
    \end{flalign*} and so we have that \begin{flalign*}
        &T \cdot O(\frac{\sigma^2}{\epsilon^2}+ \frac{1}{\epsilon}) \le O(\frac{\sigma^2}{\epsilon^2(\epsilon + B)^{\frac{p + 1}{p}}} + \frac{1}{\epsilon(\epsilon + B)^{\frac{p + 1}{p}}})
    \end{flalign*} and then applying Markov's inequality followed by a union bound yields the above oracle complexities with probability at least $\frac{5}{8}$, which finishes the proof.  
\end{proof}

\end{document}

%% file: macros.tex
\newcommand{\grad}{\nabla}

\def\ddefloop#1{\ifx\ddefloop#1\else\ddef{#1}\expandafter\ddefloop\fi}
\def\ddef#1{\expandafter\def\csname 
bb#1\endcsname{\ensuremath{\mathbb{#1}}}}
\ddefloop ABCDEFGHIJKLMNOPQRSTUVWXYZ\ddefloop
\def\ddefloop#1{\ifx\ddefloop#1\else\ddef{#1}\expandafter\ddefloop\fi}
\def\ddef#1{\expandafter\def\csname 
b#1\endcsname{\ensuremath{\mathbf{#1}}}}
\ddefloop ABCDEFGHIJKLMNOPQRSTUVWXYZ\ddefloop
\def\ddef#1{\expandafter\def\csname 
c#1\endcsname{\ensuremath{\mathcal{#1}}}}
\ddefloop ABCDEFGHIJKLMNOPQRSTUVWXYZ\ddefloop
\def\ddef#1{\expandafter\def\csname 
h#1\endcsname{\ensuremath{\widehat{#1}}}}
\ddefloop ABCDEFGHIJKLMNOPQRSTUVWXYZ\ddefloop
\def\ddef#1{\expandafter\def\csname 
hc#1\endcsname{\ensuremath{\widehat{\mathcal{#1}}}}}
\ddefloop ABCDEFGHIJKLMNOPQRSTUVWXYZ\ddefloop
\def\ddef#1{\expandafter\def\csname 
t#1\endcsname{\ensuremath{\widetilde{#1}}}}
\ddefloop ABCDEFGHIJKLMNOPQRSTUVWXYZ\ddefloop
\def\ddef#1{\expandafter\def\csname 
tc#1\endcsname{\ensuremath{\widetilde{\mathcal{#1}}}}}
\ddefloop ABCDEFGHIJKLMNOPQRSTUVWXYZ\ddefloop

\usepackage{nicefrac}

\newsavebox\CBox

\newcommand{\norm}[1]{\left \| #1 \right \|}

\newcommand{\removed}[1]{}

\usepackage{booktabs}
\usepackage{subcaption}
\usepackage{enumitem}
\DeclareSymbolFont{bbold}{U}{bbold}{m}{n}
\DeclareSymbolFontAlphabet{\mathbbold}{bbold}

%% file: higher-order-no-variance-reduction.tex
\begin{algorithm}[H]
\caption{\textbf{B}iased and \textbf{S}tochastic \textbf{P}th-\textbf{O}rder \textbf{R}egularized \textbf{T}rust \textbf{R}egion}\label{alg:novariancereduction}
    \label{alg:pth_deterministic}
\label{alg:sgd}
\begin{algorithmic}[1]
\Require Biased and stochastic oracle $(O_F^p, P_z) \in \mathcal{O}_p(F, \sigma_{1: p}, B_{1: p})$ for $F \in \mathcal{F}_p(\Delta, L_{0: p})$, Precision parameter $\epsilon$, Initial parameter $x^{(0)}$
\vspace{0.5em}
\State Find constants $\{C_i \}_{i = 1}^p$ such that (for all $n$ and all $t \ge 1$): \begin{flalign*}
    &\label{eqlabel:condition}
            \mathbb{E}||D_t^{(i)} - \nabla F^{(i)}(x^{(t)})||_{\mathrm{op}}^{\frac{p + 1}{p}} \le 2^{1/p} \cdot \left(\left(\frac{C_i \cdot \sigma_i^2}{n})^{\frac{p + 1}{2p}} + B_i^{\frac{p + 1}{p}}\right)\right)
\end{flalign*}
\State $M \gets 8 L_p$, $\eta \gets \min \{ (\epsilon + \max_i B_i)^{\frac{1}{p}}, 1 - \epsilon \}$
\State $A \gets \frac{16(p + 1)!}{M}\max \left \{1, \frac{(p!)^{\frac{p + 1}{p}}}{4 M^{\frac{1}{p}}} + 2(\frac{2p!}{M})^{\frac{1}{p}}, \frac{(p!)^{\frac{p + 1}{p}}}{4 M^{\frac{1}{p}}} + 2(\frac{2p \cdot p!}{M})^{\frac{1}{p}} \right \}$
\State $T \gets \lceil \frac{8A \Delta}{\eta^{p + 1}} \rceil $
\State Pick $n_1$ such that \begin{flalign*}
    &\max \left\{\frac{C_1 \cdot \sigma_1^2}{(\frac{\eta^{p + 1}}{8A} - B_1^{\frac{p + 1}{2p}})^{\frac{2p}{p + 1}}}, 1 \right\} \le n_1 \le \frac{(\epsilon + \max_i B_i)^{\frac{p + 1}{p}} (\max_i \sigma_i)^2}{\epsilon^3(1 + \max_i B_i)^{\frac{p + 1}{p}}} 
\end{flalign*}
\State For all $2 \le i \le p$, pick $n_i$ such that \begin{flalign*}
    &\max \left\{ \frac{C_i \sigma_i^2}{(\frac{\eta^{\frac{p^2 - 1}{p}}}{8Ap} - B_i^{\frac{p + 1}{2p}})^{\frac{2p}{p + 1}}}, 1 \right\} \le n_i \le \frac{(\epsilon + \max_i B_i)^{\frac{p + 1}{p}}(\max_i \sigma_i)^2}{\epsilon^3(1 + \max_i B_i)^{\frac{p + 1}{p}}}
\end{flalign*}

\For{$t = 0$ to $T - 1$}
    \For{$i = 1$ to $p$}
        \State Query the $i^{th}$ order oracle $n_i$ times at $x^{(t)}$ and compute \begin{flalign*}
            &D^i(x^{(t)}) = \frac{1}{n_i} \sum_{j = 1}^{n_i} \tilde{\nabla} F(x^{(t)}, z^{(t, j)}, b_i), \hspace{0.25em} z^{(t, j)} \sim P_z
        \end{flalign*}
    \EndFor
    \State Set the next point $x^{(t + 1)}$ as \begin{flalign*}
        &x^{(t + 1)} = \argmin_{y: ||y - x^{(t)}|| \le \eta} \sum_{i = 1}^p \frac{1}{i!} D^{(i)}[y - x^{(t)}]^i + \frac{M}{(p + 1)!} ||y - x^{(t)}||^{p + 1}
    \end{flalign*}
\EndFor

\State \Return $\hat{x}$ chosen uniformly at random from $\{x^{(t)} \}_{t = 1}^T$
\end{algorithmic}
\end{algorithm}

%% file: appendix_a_proofs_higher_order.tex
\begin{lemma}
\label{lemma:constantsC_i}
Given i.i.d $A_i \in \mathbb{R}^{d_1 \times \cdots \times d_m}$, where $d_1 = \ldots = d_m = d$, and $\mathbb{E}[A_i] = B$ and $\mathbb{E}[||A_i - B||^2] \le \sigma^2$, we have that  \begin{flalign*}
        &\mathbb{E}[||\frac{1}{n} \sum_{i = 1}^{n} A_i - B||_{\mathrm{op}}^2] \le \frac{C \cdot \sigma^2}{n}
    \end{flalign*} for some $d$-dependent, $n$-independent constant $C \ge 0$.
\end{lemma}

\begin{proof}
    Let $X_i = A_i - B$ and observe that
        \begin{flalign*}
        &\mathbb{E}[||\sum_{i = 1}^{n} X_i||_\mathrm{op}^2] \le \mathbb{E}_{X, X'}[||\sum_{i = 1}^{n} X_i - X_i'||^2_\mathrm{op}] \\ 
        &= \mathbb{E}_{X, X', \epsilon}[||\sum_{i = 1}^{n} \epsilon_i(X_i - X_i')||^2_\mathrm{op}] \\
        &\le 4 \mathbb{E}_{X, \epsilon}[||\sum_{i = 1}^{n} \epsilon_i X_i||^2_\mathrm{op}]
        \end{flalign*} where $(X_i')_{i = 1}^n$ is a sequence of independent copies of $(X_i)_{i = 1}^n$ and $(\epsilon_i)_{i = 1}^n$ is a sequence of Rademacher random variables. Now, take $S$ such that $S \subset \{1, \ldots, m \}$, where $|S| = \lfloor m/2 \rfloor $. We define \begin{flalign*}
        &Z_i \in \mathbb{R}^{(\prod_{k \in S} d_k) \times (\prod_{k \in S^c} d_k)}
    \end{flalign*} to be a flattened version of $X_i$, denoted by $\mathcal{F}(X_i)$. Let $D = \min \{\prod_{k \in S} d_k, \prod_{k \in S^c} d_k)$, so in this case, $D = d^{\lfloor m/2 \rfloor }$. We now prove that for any $p$, there exists a $d$-dependent constant $C_1$ such that \begin{flalign}
        &||X_i||_\mathrm{op} \le  ||Z_i||_{S_{2p}} \le C_1^2 \cdot D^{\frac{1}{2p}} \cdot   ||X_i||_\mathrm{op} 
    \end{flalign} We note that \begin{flalign*}
        &||X_i||_\mathrm{op} \\ 
        &= \sup_{||u^{(1)}|| = 1 \hspace{0.25em} \ldots \hspace{0.25em} ||u^{(m)}|| = 1} \langle X_i, u^{(1)} \otimes \ldots \otimes u^{(m)} \rangle \\
        &= \sup_{||a_1|| = ||b_1|| = 1} \langle Z_i, a_1b_1^T \rangle \\ 
        &\le \sup_{||a|| = ||b|| = 1} \langle Z_i, ab^T \rangle  \\ 
        &= ||Z_i||_\mathrm{op} =\sigma_{\max}(Z_i)
    \end{flalign*}  Note that above and below, we denote $\sigma_1, \ldots, \sigma_D$ to be the singular values of $Z_i$, where $\sigma_{\max}(Z_i) = \max_{j} \sigma_j$. We also used the fact that for any matrix $A$, $||A||_2 = \sigma_{\max}(A)$ and we defined \begin{flalign*}
        &a_1 = \bigotimes_{k \in S} u^{(k)}, b_1 = \bigotimes_{k \notin S} u^{(k)}
    \end{flalign*} Now \begin{flalign*}
        &\sigma_{\max}(Z_i) \le  (\sum_{j = 1}^D \sigma_j^{2p}(Z_i))^{\frac{1}{2p}} = ||Z_i||_{S_{2p}} \le  (D \cdot \sigma_{\max}^{2p}(Z_i))^{\frac{1}{2p}} = D^{\frac{1}{2p}} \cdot \sigma_{\max}(Z_i) \\ 
    \end{flalign*} Since \begin{flalign*}
        &||Z_i||_{\mathrm{op}} = \sup_{||a|| = ||b|| = 1} a^T Z_i b
    \end{flalign*} we expand $a$ and $b$ in their orthonormal bases as follows: \begin{flalign*}
        &a = \sum_{\alpha = 1}^{d^{\lfloor m/2 \rfloor }} a_{\alpha} e_{\alpha}, b = \sum_{\beta = 1}^{d^{\lceil m/2 \rceil }} b_{\beta} f_{\beta}
    \end{flalign*} which implies that \begin{flalign*}
        &|a^T Z_i b| = |\sum_{\alpha, \beta} a_{\alpha} b_{\beta} \langle X_i, e_{\alpha} \otimes f_{\beta} \rangle | \\ 
        &\le \sum_{\alpha, \beta} |a_{\alpha}| \cdot |b_{\beta}| \cdot |\langle X_i, e_{\alpha} \otimes f_{\beta} \rangle| \\ 
        &\le \sum_{\alpha, \beta} |a_{\alpha}| \cdot |b_{\beta}| \cdot ||X_i||_{\mathrm{op}}  \\
        &\le  ||X_i||_{\mathrm{op}} \cdot ||a||_1 \cdot ||b||_1 \\ 
        &\le C_1^2 \cdot ||X_i||_{\mathrm{op}} 
    \end{flalign*} due to the Cauchy-Schwarz inequality and since $||x||_1 \le C_1 \cdot ||x||_2$ for some constant $C_1$ for all $x$ such that $||x||_2 = 1$. This implies that \begin{flalign*}
        &D^{\frac{1}{2p}}\sigma_{\max}(Z_i) = D^{\frac{1}{2p}} \cdot ||Z_i||_{\mathrm{op}} \le D^{\frac{1}{2p}} \cdot C_1^2 \cdot ||X_i||_{\mathrm{op}}
    \end{flalign*} which proves the inequality. Observe that for $H_i$ such that $\mathcal{F}(H_i) = J_i$ and for all $\{\gamma_i \}$, it holds that \begin{flalign*}
        &\mathcal{F}(\sum_{i = 1}^n \gamma_i H_i) = \sum_{i = 1}^n \gamma_i\mathcal{F}(H_i) =  \sum_{i = 1}^n \gamma_i J_i
    \end{flalign*} and that $||X_i||_\mathrm{op}^2 \le ||Z_i||_{S_{2p}}^2$ it holds that \begin{flalign*}
        &\mathbb{E}_\epsilon[||\sum_{i = 1}^n \epsilon_i X_i||_\mathrm{op}^2] \le \mathbb{E}[|| \sum_{i = 1}^n \epsilon_i Z_i||_\mathrm{op}^2] \le \mathbb{E}_\epsilon[||\sum_{i = 1}^n \epsilon_i Z_i]||_{S_{2p}}^2] \le (\mathbb{E}_\epsilon[||\sum_{i = 1}^n \epsilon_i Z_i||_{S_{2p}}^{2p}])^{1/p}
    \end{flalign*} where the last inequality follows from an application of Hölder's inequality. By the Matrix-Khintchine inequality \citep{matrixconcentration}, we have that \begin{flalign*}
        &(\mathbb{E}_\epsilon[\sum_{i = 1}^n ||\epsilon_i Z_i||_{S_{2p}}^{2p}])^{1/p} \le (2p - 1) \cdot ||(\sum_{i = 1}^{n} Z_i^2)^{1/2}||_{S_{2p}}^2 = (2p - 1) \cdot ||\sum_{i = 1}^n Z_i^2||_{S_{2p}} \\ 
        &\le (2p - 1) \cdot \sum_{i = 1}^n ||Z_i||_{S_{2p}}^2 \\
        &\le (2p - 1) \cdot D^{1/p} C_1^2 \cdot \sum_{i = 1}^n ||X_i||^2_{\mathrm{op}} \\
    \end{flalign*} Taking $p = 1$, we have that \begin{flalign*}
        &\mathbb{E}_\epsilon[\sum_{i = 1}^n ||\epsilon_i Z_i||_{S_2}^2] \le DC_1^2 \cdot \sum_{i = 1}^n ||X_i||_\mathrm{op}^2 
    \end{flalign*} and when taking expectation with respect to $X$, we have that \begin{flalign*}
        &\mathbb{E}[||\sum_{i = 1}^n \epsilon_i X_i||_\mathrm{op}^2] \le DC_1^2 \cdot \sum_{i = 1}^n \mathbb{E}||X_i||_\mathrm{op}^2 \le DC_1^2\sigma^2 n \le d^{m/2} C_1^2 \sigma^2n
    \end{flalign*} Normalizing by $n^2$ gives a final bound of \begin{flalign*}
        \frac{4d^{m/2}C_1^2 \sigma^2}{n}
    \end{flalign*} which finishes the proof. 
\end{proof}

\begin{lemma}
    For all integers $p \ge 1$, for all $i \in \{1, \ldots, p \}$ and all $t \ge 1$, there exists a $d$-dependent, $n_i$-independent constant $C \ge 0$ such that \begin{flalign*}
        &\mathbb{E}||D_t^{(i)} - \nabla F^{(i)} (x^{(t)})||_{\mathrm{op}}^{\frac{p + 1}{p}} \le 2^{1/p} \cdot \left(\left(\frac{C \cdot \sigma_i^2}{n_i}\right)^{\frac{p + 1}{2p}} + B_i^{\frac{p + 1}{p}}\right) 
    \end{flalign*}
\end{lemma}
\begin{proof}
    First, we can say that \begin{flalign*}
        &\mathbb{E}[||D_t^{(i)} - \nabla F^i(x^{(t)})||^{\frac{p + 1}{p}}] \\ 
        &= \mathbb{E}[||D_t^{(i)} - \nabla F^i(x^{(t)}) - b_i(x^{(t)}) + b_i(x^{(t)})||^{\frac{p + 1}{p}}] \\ 
        &\le 2^{1/p}\cdot (\mathbb{E}[||D_t^{(i)} - \nabla F^i(x^{(t)}) - b_i(x^{(t)})||^{\frac{p + 1}{p}}] + \mathbb{E}[||b_i(x^{(t)})||^{\frac{p + 1}{p}}]) \\
        &\le 2^{1/p}\cdot (\mathbb{E}[||D_t^{(i)} - \nabla F^i(x^{(t)}) - b_i(x^{(t)})||^{\frac{p + 1}{p}}] + B_i^{\frac{p + 1}{p}}) 
    \end{flalign*} Notice that for any $r \in [1, 2]$, we can have that \begin{flalign*}
        &\mathbb{E}[||D_t^{(i)} - \nabla F^i(x^{(t)}) - b_i(x^{(t)})||_\mathrm{op}^{r}] \\  
        &= \mathbb{E}[||\frac{1}{n_i} \sum_{j = 1}^{n_i} \tilde{\nabla}^i F(x^{(t)}, z^{(t, j)}) - \nabla F^i(x^{(t)}) - b_i(x^{(t)})||_\mathrm{op}^{r}] \\ 
        &\le (\mathbb{E}[||\frac{1}{n_i} \sum_{j = 1}^{n_i} \tilde{\nabla}^i F(x^{(t)}, z^{(t, j)}) - \nabla F^i(x^{(t)}) - b_i(x^{(t)})||_\mathrm{op}^{2}])^{r/2} \\
        &\le (\frac{C \cdot \sigma_i^2}{n_i})^{r/2}  
    \end{flalign*} where we have used Lyapunov's inequality and the result from lemma \ref{lemma:constantsC_i}. Thus, we have a final bound of \begin{flalign*}
        &2^{1/p}\cdot \left(\left(\frac{C \cdot \sigma_i^2}{n_i}\right)^{\frac{p + 1}{2p}} + B_i^{\frac{p + 1}{p}}\right)
    \end{flalign*} which finishes the proof.
\end{proof}

\begin{lemma}
\label{lemma:diffxandy}
Given a function $F \in \mathcal{F}_p(\Delta, L_{1: p})$, let \begin{flalign*}
        &m_x(y) = F(x) + \langle D^{(1)}, y - x \rangle + \sum_{i = 2}^p \frac{1}{i!} D^{(i)}(x)[y - x]^i + \frac{M}{(p + 1)!} ||y - x||^{p + 1}
    \end{flalign*} and let
        $y \in \argmin_{z: ||z - x|| \le \eta} m_x(z)$ for $0 \le \eta < 1$. Then, for all $M \ge 8 L_p$, we have that \begin{flalign*}
        &F(x) - F(y) > \frac{M}{8(p + 1)!} ||y - x||^{p + 1} - (\frac{2p!}{M})^{\frac{1}{p }} \cdot ||\nabla F(x) - D^{(1)}(x)||^{\frac{p + 1}{p}} \\ 
        &- \frac{1}{2} \sum_{i = 2}^p (\frac{2p \cdot p!}{M})^{\frac{1}{p }} \cdot ||\nabla^i F(x) - D^i(x)||_{\mathrm{op}}^{\frac{p + 1}{p}} \cdot \eta^{\frac{p + 1}{p}}
    \end{flalign*}
\end{lemma}

\begin{proof}
    We have that $F(y) - F(x)$ \begin{flalign*} 
        &\le F(x) + \langle \nabla F(x), y - x \rangle + \sum_{i = 2}^{p} \frac{1}{i!} \nabla^iF(x) [y - x]^i + \frac{L_p}{(p + 1)!} ||y - x||^{p + 1} - F(x) \\ 
        &\le m_x(y) + \langle \nabla F(x) - D^{(1)}(x), y - x \rangle + \sum_{i = 2}^{p} \frac{1}{i!}(\nabla^i F(x) - D^{(i)}(x))[y - x]^i + \frac{L_p - M}{(p + 1)!} ||y - x||^{p + 1} - m_x(x) \\
        &\le \langle \nabla F(x) - g, y - x \rangle + \sum_{i = 2}^{p} \frac{1}{i!}(\nabla^i F(x) - D^{(i)}(x))[y - x]^i + \frac{L_p - M}{(p + 1)!} ||y - x||^{p + 1} \\ 
        &\le  -\frac{7M}{8(p + 1)!}||y - x||^{p + 1} + ||\nabla F(x) - g|| \cdot ||y - x|| \\ 
        &+ \sum_{i = 2}^{p} \frac{1}{i!} ||\nabla^i F(x)[y - x, :, \ldots, :] - D^{(i)}(x)[y - x, :, \ldots, :] ||_\mathrm{op} \cdot ||y - x||
    \end{flalign*} since $||y - x|| \le \eta$ and since $\eta < 1$, we have that $||y - x||^{i - 1} \le ||y - x||$ for $i \ge 2$. By Young's inequality, we have that \begin{flalign*}
        &||\nabla F(x) - D^{(1)}(x)|| \cdot ||y - x|| \le ((\frac{2p!}{M})^{\frac{1}{p}} \cdot ||\nabla F(x) - D^{(1)}(x)||^{\frac{p + 1}{p}} \cdot \frac{p}{p + 1}) + (\frac{||y - x||^{p + 1}}{(p + 1)} \cdot \frac{M}{2 p!}) \\
        &= (\frac{2p!}{M})^{\frac{1}{p}}  (\frac{p \cdot ||\nabla F(x) - D^{(1)}(x)||^{\frac{p + 1}{p}}}{p + 1}) + \frac{M ||y - x||^{p + 1}}{2 (p + 1)!}
    \end{flalign*} and \begin{flalign*}
        &||\nabla^i F(x)[y - x, :, \ldots, :] - D^{(i)}(x)[y - x, :, \ldots, :] ||_\mathrm{op} \cdot ||y - x|| \\ 
        &\le (\frac{2p \cdot p!}{M})^{\frac{1}{p}}  \frac{p  \cdot ||\nabla^i F(x)[y - x, :, \ldots, :] - D^{(i)}(x)[y - x, :, \ldots, :]||_\mathrm{op}^{\frac{p + 1}{p}}}{p + 1} + \frac{M  ||y - x||^{p + 1}}{(p + 1) \cdot (2p \cdot p!)} \\ 
        &= (\frac{2p \cdot p!}{M})^{\frac{1}{p }}  \frac{p  \cdot ||\nabla^i F(x)[y - x, :, \ldots, :] - D^{(i)}(x)[y - x, :, \ldots, :]||_\mathrm{op}^{\frac{p + 1}{p}}}{p + 1} + \frac{M  ||y - x||^{p + 1}}{2p \cdot (p + 1)!}
    \end{flalign*} which implies that \begin{flalign*}
        &-\frac{7M}{8(p + 1)!}||y - x||^{p + 1} + ||\nabla F(x) - D^{(1)}(x)|| \cdot ||y - x|| \\ 
        &+ \sum_{i = 2}^{p} \frac{1}{i!} ||\nabla^i F(x)[y - x, :, \ldots, :] - D^{(i)}(x)[y - x, :, \ldots, :] ||_\mathrm{op} \cdot ||y - x|| \\ 
        &\le -\frac{7M}{8(p + 1)!}||y - x||^{p + 1} + (\frac{2p!}{M})^{\frac{1}{p}} \cdot \frac{p}{p + 1} \cdot ||\nabla F(x) - g||^{\frac{p + 1}{p}} + \frac{M ||y - x||^{p + 1}}{2 (p + 1)!} \\ 
        &+ \sum_{i = 2}^{p} \frac{1}{i!} \cdot [(\frac{2p \cdot p!}{M})^{\frac{1}{p}}  \frac{p  \cdot ||\nabla^i F(x)[y - x, :, \ldots, :] - D^{(i)}(x)[y - x, :, \ldots, :]||_\mathrm{op}^{\frac{p + 1}{p}}}{p + 1} + \frac{M  ||y - x||^{p + 1}}{2p \cdot (p + 1)!}] \\
        &\le -\frac{3M}{8(p + 1)!}||y - x||^{p + 1} + (\frac{2p!}{M})^{\frac{1}{p }} \cdot \frac{p}{p + 1} \cdot ||\nabla F(x) - D^{(1)}(x)||^{\frac{p + 1}{p}} \\  
        &+ \sum_{i = 2}^{p} \frac{1}{i!}[(\frac{2p \cdot p!}{M})^{\frac{1}{p }}  \frac{p \cdot ||\nabla^i F(x) - D^{(i)}(x)||_\mathrm{op}^{\frac{p + 1}{p}} \cdot ||y - x||^{\frac{p + 1}{p}}}{p + 1} + \frac{M \cdot ||y - x||^{p + 1}}{2p \cdot (p + 1)!}] \\
        &< -\frac{3M}{8(p + 1)!}||y - x||^{p + 1} + (\frac{2p!}{M})^{\frac{1}{p }} \cdot \frac{p}{p + 1} \cdot ||\nabla F(x) - D^{(1)}(x)||^{\frac{p + 1}{p}} \\ 
        &+ \sum_{i = 2}^{p} \frac{1}{2}[(\frac{2p \cdot p!}{M})^{\frac{1}{p }}  \frac{p \cdot ||\nabla^i F(x) - D^{(i)}(x)||^{\frac{p + 1}{p}} \cdot ||y - x||^{\frac{p + 1}{p}}}{p + 1} + \frac{M \cdot ||y - x||^{p + 1}}{2p \cdot (p + 1)!}] \\ 
    \end{flalign*} We can further bound this expression by \begin{flalign*}
        &< -\frac{3M}{8(p + 1)!}||y - x||^{p + 1} + (\frac{2p!}{M})^{\frac{1}{p }} \cdot \frac{p}{p + 1} \cdot ||\nabla F(x) - D^{(1)}(x)||^{\frac{p + 1}{p}} \\
        &+ \frac{1}{2}\sum_{i = 2}^{p} [(\frac{2p \cdot p!}{M})^{\frac{1}{p }}  \cdot ||\nabla^i F(x) - D^{(i)}(x)||_\mathrm{op}^{\frac{p + 1}{p}} \cdot ||y - x||^{\frac{p + 1}{p}} + \frac{M \cdot ||y - x||^{p + 1}}{2p \cdot (p + 1)!}] \\
        &< -\frac{M}{8(p + 1)!}||y - x||^{p + 1} + (\frac{2p!}{M})^{\frac{1}{p }} \cdot \frac{p}{p + 1} \cdot ||\nabla F(x) - D^{(1)}(x)||^{\frac{p + 1}{p}} \\ 
        &+ \frac{1}{2}\sum_{i = 2}^{p} (\frac{2p \cdot p!}{M})^{\frac{1}{p }} \cdot ||\nabla^i F(x) - D^{(i)}(x)||_\mathrm{op}^{\frac{p + 1}{p}} \cdot \eta^{\frac{p + 1}{p}} \\ 
        &< -\frac{M}{8(p + 1)!}||y - x||^{p + 1} + (\frac{2p!}{M})^{\frac{1}{p }} \cdot ||\nabla F(x) - D^{(1)}(x)||^{\frac{p + 1}{p}} + \frac{1}{2}\sum_{i = 2}^{p} (\frac{2p \cdot p!}{M})^{\frac{1}{p }} \cdot ||\nabla^i F(x) - D^{(i)}(x)||_\mathrm{op}^{\frac{p + 1}{p}} \cdot \eta^{\frac{p + 1}{p}}
    \end{flalign*} which finishes the proof. 
\end{proof}
\newpage 
\begin{lemma}
    Given a function $F \in \mathcal{F}_p(\Delta, L_{1: p})$, let $y \in \argmin_{z: ||z - x|| \le \eta}$ $m_x(z)$, where \begin{flalign*}
        &m_x(y) = F(x) + \langle D^{(1)}, y - x \rangle + \sum_{i = 2}^p \frac{1}{i!} D^{(i)}[y - x]^i + \frac{M}{(p + 1)!} ||y - x||^{p + 1}
    \end{flalign*} for $M \ge 8L_p$ and $0 \le \eta < 1$. It holds that: 
     \begin{flalign*}
        &\boldsymbol{1}[||\nabla F(y)|| \ge \frac{9M}{8p!} \eta^p] \le \frac{1}{\eta^p}||y - x||^p + \frac{p!}{M \eta^p}(||\nabla F(x) - D^{(1)}(x)|| + \sum_{i = 1}^{p - 1} \frac{1}{i!} ||\nabla^{i + 1} F(x) - D^{(i + 1)}(x)||_\mathrm{op}\cdot \eta^i)
    \end{flalign*}
\end{lemma}

\begin{proof}
    We have that \begin{flalign*}
        &||\nabla F(y)|| \\  
        &\le ||\nabla F(y) - \sum_{i = 0}^{p - 1} \frac{1}{i!} \nabla^{i + 1} F(x)[y - x]^i|| + ||\sum_{i = 0}^{p - 1} \frac{1}{i!} \nabla^{i + 1}F(x)[y - x]^i|| \\ 
        &\le \frac{L_p}{p!} ||y - x||^p + ||\nabla F(x) + \sum_{i = 1}^{p - 1} \frac{1}{i!} \nabla^{i + 1}F(x)[y - x]^i|| \\ 
        &\le \frac{L_p}{p!} ||y - x||^p + ||\nabla F(x) - D^{(1)}(x)|| + ||\sum_{i = 1}^{p - 1} \frac{1}{i!} [\nabla^{i + 1} F(x)[y - x]^i - D^{(i + 1)}(x)[y - x]^i]|| \\
        &+ ||D^{(1)}(x) + \sum_{i = 1}^{p - 1} \frac{1}{i!} D^{(i + 1)}[y - x]^i|| \\
        &\le \frac{L_p}{p!} ||y - x||^p + ||\nabla F(x) - D^{(1)}(x)|| + \sum_{i = 1}^{p - 1} \frac{1}{i!} ||\nabla^{i + 1} F(x) - D^{(i + 1)}(x)|| \cdot ||y - x||^i \\ 
        &+ ||D^{(1)}(x) + \sum_{i = 1}^{p - 1} \frac{1}{i!} D^{(i + 1)}[y - x]^i|| \\
        &\le \frac{L_p + M}{p!} ||y - x||^p + ||\nabla F(x) - D^{(1)}(x)|| + \sum_{i = 1}^{p - 1} \frac{1}{i!} ||\nabla^{i + 1} F(x) - D^{(i + 1)}(x)|| \cdot ||y - x||^i \\ 
        &\le \frac{L_p + M}{p!} ||y - x||^p + ||\nabla F(x) - D^{(1)}(x)|| + \sum_{i = 1}^{p - 1} \frac{1}{i!} ||\nabla^{i + 1} F(x) - D^{(i + 1)}(x)|| \cdot \eta^i
    \end{flalign*} since under first order optimality conditions for $y \in \argmin_{z} m_x(z)$, we have that   \begin{flalign*}
        &D^{(1)}(x) + \sum_{i = 1}^{p - 1} \frac{1}{i!} D^{(i + 1)}[y - x]^i + \frac{M}{p!}||y - x||^{p + 1}(x - y)= 0
    \end{flalign*} We now have that \begin{flalign*}
        &||y - x||^p \ge \frac{p!}{L_p + M}(||\nabla F(y)|| - ||\nabla F(x) - D^{(1)}(x)|| - \sum_{i = 1}^{p - 1} \frac{1}{i!} ||\nabla^{i + 1} F(x) - D^{(i + 1)}(x)|| \cdot \eta^i) \\  
        &\ge \min \{\eta^p, \frac{p!}{L_p + M}(||\nabla F(y)|| - ||\nabla F(x) - D^{(1)}(x)|| - \sum_{i = 1}^{p - 1} \frac{1}{i!} ||\nabla^{i + 1} F(x) - D^{(i + 1)}(x)|| \cdot \eta^i)) \} \\
        &\ge \min \{\eta^p, \frac{p!}{L_p + M} ||\nabla F(y)|| \}  - \frac{p!}{L_p + M} ||\nabla F(x) - D^{(1)}(x)|| - \frac{p!}{L_p + M} \sum_{i = 1}^{p - 1} \frac{1}{i!} ||\nabla^{i + 1} F(x) - D^{(i + 1)}(x)|| \cdot \eta^i
    \end{flalign*} and since $L_p \le \frac{M}{8}$ and $\frac{M}{L_p + M} < 1$, we have that \begin{flalign*}
        &M||y - x||^p \ge \min \{M\eta^p, \frac{Mp!}{L_p + M} ||\nabla F(y)|| \}  - \frac{Mp!}{L_p + M} ||\nabla F(x) - D^{(1)}(x)|| \\ 
        &- \frac{Mp!}{L_p + M} \sum_{i = 1}^{p - 1} \frac{1}{i!} ||\nabla^{i + 1} F(x) - D^{(i + 1)}(x)|| \cdot \eta^i \\ 
        &> \min \{M\eta^p, \frac{8p!}{9} ||\nabla F(y)|| \} - p!(||\nabla F(x) - D^{(1)}(x)|| + \sum_{i = 1}^{p - 1} \frac{1}{i!} ||\nabla^{i + 1} F(x) - D^{(i + 1)}(x)|| \cdot \eta^i)
    \end{flalign*} which means that \begin{flalign*} 
        &\min \{M\eta^p, \frac{8p!}{9} ||\nabla F(y)|| \} < M||y - x||^p + p!(||\nabla F(x) - D^{(1)}(x)|| + \sum_{i = 1}^{p - 1} \frac{1}{i!} ||\nabla^{i + 1} F(x) - D^{(i + 1)}(x)|| \cdot \eta^i)
    \end{flalign*} which then implies that (since for all $a, b \ge 0$, $a \boldsymbol{1}[b \ge a ] \le \min \{a, b \})$ 
    \begin{flalign*}
        &M \eta^p \cdot \boldsymbol{1}[||\nabla F(y)|| \ge \frac{9M}{8p!} \eta^p] \le M||y - x||^p + p!(||\nabla F(x) - D^{(1)}(x)|| + \sum_{i = 1}^{p - 1} \frac{1}{i!} ||\nabla^{i + 1} F(x) - D^{(i + 1)}(x)|| \cdot \eta^i) \\ 
        &\implies \boldsymbol{1}[||\nabla F(y)|| \ge \frac{9M}{8p!} \eta^p] \le \frac{1}{\eta^p}||y - x||^p + \frac{p!}{M \eta^p}(||\nabla F(x) - D^{(1)}(x)|| + \sum_{i = 1}^{p - 1} \frac{1}{i!} ||\nabla^{i + 1} F(x) - D^{(i + 1)}(x)|| \cdot \eta^i) 
     \end{flalign*} which proves the lemma. 
\end{proof}

\begin{lemma}
    \label{lemma:randomvarderivatives}
    We now consider the setting where the derivative estimates $D^i$ are random variables. For all $F \in \mathcal{F}_p(\Delta, L_{1: p})$, it holds that \begin{flalign*}
        &\mathbb{E}[F(x) - F(y)] \ge  \frac{M \eta^{p + 1}}{16(p + 1)!} \Pr(||\nabla F(y)|| \ge \frac{9M}{8p!} \eta^p) - [\frac{(p!)^{\frac{p + 1}{p}}}{4 \sqrt[p]{M}} + 2(\frac{2p!}{M})^{\frac{1}{p + 1}}] \cdot [(\frac{\sigma_1^2}{n_1})^{\frac{p + 1}{2p}} + B_1^{\frac{p + 1}{2p}}] \\ 
        &- \eta^{\frac{p + 1}{p}}(\frac{(p!)^{\frac{p + 1}{p}}}{4 \sqrt[p]{M}} +  (\frac{2p \cdot p!}{M})^{\frac{1}{p + 1}}) \sum_{i = 2}^p [(\frac{\sigma_i^2}{n_i})^{\frac{p + 1}{2p}} + B_i^{\frac{p + 1}{2p}}]
    \end{flalign*}
\end{lemma}

\begin{proof}
    First, we note that 
    \begin{flalign*}
        &\boldsymbol{1}[||\nabla F(y)|| \ge \frac{9M}{8p!} \eta^p] \\  
        &\le (\frac{1}{\eta^p}||y - x||^p + \frac{p!}{M \eta^p}(||\nabla F(x) - g|| + \sum_{i = 1}^{p - 1} \frac{1}{i!} ||\nabla^{i + 1} F(x) - D^{i + 1}(x)|| \cdot \eta^i))^{\frac{p + 1}{p}} \\ 
        &\le \frac{2^{1/p}}{\eta^{p + 1}} ||y - x||^{p + 1} + 2^{1/p} \cdot (\frac{p!}{M \eta^p})^{\frac{p + 1}{p}} \cdot (\sum_{i = 0}^{p - 1} \frac{1}{i!} ||\nabla^{i + 1} F(x) - D^{i + 1}(x)|| \cdot \eta^i)^{\frac{p + 1}{p}} \\
        &< \frac{2}{\eta^{p + 1}} ||y - x||^{p + 1} + \frac{2 (p!)^{\frac{p + 1}{p}}}{M^{\frac{p + 1}{p}} \eta^{p + 1}} \cdot (\sum_{i = 0}^{p - 1} ||\nabla^{i + 1} F(x) - D^{i + 1}(x)|| \cdot \eta^i)^{\frac{p + 1}{p}}
    \end{flalign*} where we used the fact that for any $a_i \ge 0$, we have that \begin{flalign*}
        &(\sum_{i = 1}^n a_i)^{\frac{p + 1}{p}} \le n^{\frac{1}{p}} \sum_{i = 1}^n a_i^{\frac{p + 1}{p}}
    \end{flalign*} which follows from an application of Hölder's inequality. We can now continue to bound the above expression by \begin{flalign*}
        &\frac{2}{\eta^{p + 1}} ||y - x||^{p + 1} + \frac{2 (p!)^{\frac{p + 1}{p}}}{M^{\frac{p + 1}{p}} \eta^{p + 1}} \cdot p^{1/p} \cdot \sum_{i = 0}^{p - 1} ||\nabla^{i + 1} F(x) - D^{i + 1}(x)||^{\frac{p + 1}{p}} \cdot \eta^{\frac{i(p + 1)}{p}} \\ 
        &< \frac{2}{\eta^{p + 1}} ||y - x||^{p + 1} + \frac{4 (p!)^{\frac{p + 1}{p}}}{M^{\frac{p + 1}{p}} \eta^{p + 1}} \cdot \sum_{i = 0}^{p - 1} ||\nabla^{i + 1} F(x) - D^{i + 1}(x)||^{\frac{p + 1}{p}} \cdot \eta^{\frac{i(p + 1)}{p}} 
    \end{flalign*} where we used the fact that for all $p \ge 1$, $p^{1/p} < 2$. Taking expectations on each side, we have that \begin{flalign*}
        &\mathbb{E}[||y - x||^{p + 1}] \ge \frac{\eta^{p + 1}}{2}\Pr(||\nabla F(y)|| \ge \frac{9M}{8p!} \eta^p) - \frac{2 (p!)^{\frac{p + 1}{p}}}{M^{\frac{p + 1}{p}} } \cdot \sum_{i = 0}^{p - 1} \mathbb{E}[||\nabla^{i + 1} F(x) - D^{i + 1}(x)||^{\frac{p + 1}{p}} \cdot \eta^{\frac{i(p + 1)}{p}}]
    \end{flalign*} We also have that \begin{flalign*}
        &\mathbb{E}[F(x) - F(y)] \\ 
        &> \frac{M}{8(p + 1)!} ||y - x||^{p + 1} - (\frac{2p!}{M})^{\frac{1}{p }} \cdot ||\nabla F(x) - D^{(1)}(x)||^{\frac{p + 1}{p}} \\
        &- \frac{1}{2} \sum_{i = 2}^p (\frac{2p \cdot p!}{M})^{\frac{1}{p }} \cdot ||\nabla^i F(x) - D^i(x)||_{\mathrm{op}}^{\frac{p + 1}{p}} \cdot \eta^{\frac{p + 1}{p}} \\
        &\ge \frac{M \eta^{p + 1}}{16(p + 1)!} \Pr(||\nabla F(y)|| \ge \frac{9M}{8p!} \eta^p) - \frac{2(p!)^{\frac{p + 1}{p}}}{\sqrt[p]{M} \cdot 8(p + 1)!} \cdot \sum_{i = 0}^{p - 1} \mathbb{E}[||\nabla^{i + 1} F(x) - D^{i + 1}(x)||^{\frac{p + 1}{p}} \cdot \eta^{\frac{i(p + 1)}{p}}] \\ 
        &- (\frac{2p!}{M})^{\frac{1}{p}} \cdot \mathbb{E}[||\nabla F(x) - D^{(1)}(x)||^{\frac{p + 1}{p}}] - \frac{1}{2}\sum_{i = 2}^{p} (\frac{2p \cdot p!}{M})^{\frac{1}{p }} \cdot \mathbb{E}[||\nabla^i F(x) - D^i(x)||^{\frac{p + 1}{p}}] \cdot \eta^{\frac{p + 1}{p}} \\
        &\ge \frac{M \eta^{p + 1}}{16(p + 1)!} \Pr(||\nabla F(y)|| \ge \frac{9M}{8p!} \eta^p) - \frac{(p!)^{\frac{p + 1}{p}}}{8 \sqrt[p]{M}} \cdot \sum_{i = 0}^{p - 1} \mathbb{E}[||\nabla^{i + 1} F(x) - D^{i + 1}(x)||^{\frac{p + 1}{p}} \cdot \eta^{\frac{i(p + 1)}{p}}] \\ 
        &- (\frac{2p!}{M})^{\frac{1}{p }} \cdot \mathbb{E}[||\nabla F(x) - D^{(1)}(x)||^{\frac{p + 1}{p}}] - \frac{1}{2}\sum_{i = 2}^{p} (\frac{2p \cdot p!}{M})^{\frac{1}{p }} \cdot \mathbb{E}[||\nabla^i F(x) - D^i(x)||^{\frac{p + 1}{p}}] \cdot \eta^{\frac{p + 1}{p}} \\ 
        &= \frac{M \eta^{p + 1}}{16(p + 1)!} \Pr(||\nabla F(y)|| \ge \frac{9M}{8p!} \eta^p) - [\frac{(p!)^{\frac{p + 1}{p}}}{8 \sqrt[p]{M}} + (\frac{2p!}{M})^{\frac{1}{p }}] \cdot \mathbb{E}[||\nabla F(x) - D^{(1)}(x)||^{\frac{p + 1}{p}}] \\
        &- \frac{(p!)^{\frac{p + 1}{p}}}{8 \sqrt[p]{M}} \cdot \sum_{i = 1}^{p - 1} \mathbb{E}[||\nabla^{i + 1} F(x) - D^{i + 1}(x)||^{\frac{p + 1}{p}} \cdot \eta^{\frac{i(p + 1)}{p}}] - \frac{1}{2}\sum_{i = 2}^{p} (\frac{2p \cdot p!}{M})^{\frac{1}{p }} \cdot \mathbb{E}[||\nabla^i F(x) - D^i(x)||^{\frac{p + 1}{p}}] \cdot \eta^{\frac{p + 1}{p}} \\
        &= \frac{M \eta^{p + 1}}{16(p + 1)!} \Pr(||\nabla F(y)|| \ge \frac{9M}{8p!} \eta^p) - [\frac{(p!)^{\frac{p + 1}{p}}}{8 \sqrt[p]{M}} + (\frac{2p!}{M})^{\frac{1}{p }}] \cdot \mathbb{E}[||\nabla F(x) - D^{(1)}(x)||^{\frac{p + 1}{p}}] \\ 
        &- \frac{(p!)^{\frac{p + 1}{p}}}{8 \sqrt[p]{M}} \cdot \sum_{i = 2}^{p} \mathbb{E}[||\nabla^{i} F(x) - D^{i}(x)||^{\frac{p + 1}{p}} \cdot \eta^{\frac{(i - 1)(p + 1)}{p}}] - \frac{1}{2}\sum_{i = 2}^{p} (\frac{2p \cdot p!}{M})^{\frac{1}{p }} \cdot \mathbb{E}[||\nabla^i F(x) - D^i(x)||^{\frac{p + 1}{p}}] \cdot \eta^{\frac{p + 1}{p}} \\ 
        &\ge \frac{M \eta^{p + 1}}{16(p + 1)!} \Pr(||\nabla F(y)|| \ge \frac{9M}{8p!} \eta^p) - [\frac{(p!)^{\frac{p + 1}{p}}}{8 \sqrt[p]{M}} + (\frac{2p!}{M})^{\frac{1}{p}}] \cdot \mathbb{E}[||\nabla F(x) - g||^{\frac{p + 1}{p}}] \\ 
        &- \frac{(p!)^{\frac{p + 1}{p}}}{8 \sqrt[p]{M}} \cdot \sum_{i = 2}^{p} \mathbb{E}[||\nabla^{i} F(x) - D^{i}(x)||^{\frac{p + 1}{p}} \cdot \eta^{\frac{(p + 1)}{p}}] - \frac{1}{2}\sum_{i = 2}^{p} (\frac{2p \cdot p!}{M})^{\frac{1}{p }} \cdot \mathbb{E}[||\nabla^i F(x) - D^i(x)||^{\frac{p + 1}{p}}] \cdot \eta^{\frac{p + 1}{p}} 
    \end{flalign*} where the last step follows from the fact that $\eta \le  1$. We further lower bound this expression by \begin{flalign*}
        &\frac{M \eta^{p + 1}}{16(p + 1)!} \Pr(||\nabla F(y)|| \ge \frac{9M}{8p!} \eta^p) - [\frac{(p!)^{\frac{p + 1}{p}}}{8 \sqrt[p]{M}} + (\frac{2p!}{M})^{\frac{1}{p }}] \cdot \mathbb{E}[||\nabla F(x) - D^{(1)}(x)||^{\frac{p + 1}{p}}] \\ 
        &- \eta^{\frac{p + 1}{p}}(\frac{(p!)^{\frac{p + 1}{p}}}{8 \sqrt[p]{M}} + \frac{1}{2} (\frac{2p \cdot p!}{M})^{\frac{1}{p }}) \sum_{i = 2}^p \mathbb{E}[||\nabla^i F(x) - D^i(x)||^{\frac{p + 1}{p}}] \\ 
        &\ge \frac{M \eta^{p + 1}}{16(p + 1)!} \Pr(||\nabla F(y)|| \ge \frac{9M}{8p!} \eta^p) - [\frac{(p!)^{\frac{p + 1}{p}}}{8 \sqrt[p]{M}} + (\frac{2p!}{M})^{\frac{1}{p }}] \cdot 2^{1/p}[(\frac{C_1 \cdot \sigma_1^2}{n_1})^{\frac{p + 1}{2p}} + B_1^{\frac{p + 1}{2p}}] \\ 
        &- \eta^{\frac{p + 1}{p}}(\frac{(p!)^{\frac{p + 1}{p}}}{8 \sqrt[p]{M}} + \frac{1}{2} (\frac{2p \cdot p!}{M})^{\frac{1}{p }}) \sum_{i = 2}^p 2^{1/p} \cdot [(\frac{C_i \cdot \sigma_i^2}{n_i})^{\frac{p + 1}{2p}} + B_i^{\frac{p + 1}{2p}}] \\
        &\ge \frac{M \eta^{p + 1}}{16(p + 1)!} \Pr(||\nabla F(y)|| \ge \frac{9M}{8p!} \eta^p) - [\frac{(p!)^{\frac{p + 1}{p}}}{4 \sqrt[p]{M}} + 2(\frac{2p!}{M})^{\frac{1}{p }}] \cdot [(\frac{C_1 \cdot \sigma_1^2}{n_1})^{\frac{p + 1}{2p}} + B_1^{\frac{p + 1}{2p}}] \\ 
        &- \eta^{\frac{p + 1}{p}}(\frac{(p!)^{\frac{p + 1}{p}}}{4 \sqrt[p]{M}} +  (\frac{2p \cdot p!}{M})^{\frac{1}{p }}) \sum_{i = 2}^p [(\frac{C_i \cdot \sigma_i^2}{n_i})^{\frac{p + 1}{2p}} + B_i^{\frac{p + 1}{2p}}]
    \end{flalign*} where we used the fact that $2^{1/p} \le 2$ for all $p \ge 1$, finishing the proof. 
\end{proof}
\newpage 
\begin{lemma}
    \label{lemma:gradprobbound}
    Let $F \in \mathcal{F}_p(\Delta, L_{1: p})$ be given. Then, if the derivative estimates $D^i$ are random variables, 
    it holds that \begin{flalign*}
        &\Pr(||\nabla F(\hat{x})|| \ge \frac{9M}{8p!} \eta^p) \le \frac{16(p + 1)!}{M \eta^{p + 1}T}\Delta + \frac{16(p + 1)!}{M \eta^{p + 1}}[\frac{(p!)^{\frac{p + 1}{p}}}{4 \sqrt[p]{M}} + 2(\frac{2p!}{M})^{\frac{1}{p }}] \cdot [(\frac{C_1 \cdot \sigma_1^2}{n_1})^{\frac{p + 1}{2p}} + B_1^{\frac{p + 1}{2p}}] \\ 
        &+ \frac{16(p + 1)!}{M \eta^{\frac{p^2 - 1}{p}}} (\frac{(p!)^{\frac{p + 1}{p}}}{4 \sqrt[p]{M}} +  (\frac{2p \cdot p!}{M})^{\frac{1}{p}}) \sum_{i = 2}^p [(\frac{C_i \cdot \sigma_i^2}{n_i})^{\frac{p + 1}{2p}} + B_i^{\frac{p + 1}{2p}}]
    \end{flalign*}
\end{lemma}

\begin{proof}
    From lemma \ref{lemma:randomvarderivatives}, we have that \begin{flalign*}
        &\mathbb{E}[F(x^{(t)}) - F(x^{(t + 1)})] \\ 
        &\ge \frac{M \eta^{p + 1}}{16(p + 1)!} \Pr(||\nabla F(x^{(t + 1)})|| \ge \frac{9M}{8p!} \eta^p) - [\frac{(p!)^{\frac{p + 1}{p}}}{4 \sqrt[p]{M}} + 2(\frac{2p!}{M})^{\frac{1}{p }}] \cdot [(\frac{C_1 \cdot \sigma_1^2}{n_1})^{\frac{p + 1}{2p}} + B_1^{\frac{p + 1}{2p}}] \\
        &- \eta^{\frac{p + 1}{p}}(\frac{(p!)^{\frac{p + 1}{p}}}{4 \sqrt[p]{M}} +  (\frac{2p \cdot p!}{M})^{\frac{1}{p}}) \sum_{i = 2}^p [(\frac{C_i \cdot \sigma_i^2}{n_i})^{\frac{p + 1}{2p}} + B_i^{\frac{p + 1}{2p}}] 
    \end{flalign*} Telescoping this recurrence from $t = 1$ to $T$ gives \begin{flalign*}
        &\mathbb{E}[F(x^{(1)}) - F(x^{(T + 1)})] \\ 
    &\ge \frac{M \eta^{p + 1}}{16 (p + 1)!} \cdot T \cdot (\frac{1}{T} \sum_{t = 1}^T \Pr(||\nabla F(x^{(t + 1)})|| \ge \frac{9M}{8p!} \eta^p)) - T[\frac{(p!)^{\frac{p + 1}{p}}}{4 \sqrt[p]{M}} + 2(\frac{2p!}{M})^{\frac{1}{p }}] \cdot [(\frac{C_1 \cdot \sigma_1^2}{n_1})^{\frac{p + 1}{2p}} + B_1^{\frac{p + 1}{2p}}] \\
    &- T\eta^{\frac{p + 1}{p}}(\frac{(p!)^{\frac{p + 1}{p}}}{4 \sqrt[p]{M}} +  (\frac{2p \cdot p!}{M})^{\frac{1}{p }}) \sum_{i = 2}^p [(\frac{C_i \cdot \sigma_i^2}{n_i})^{\frac{p + 1}{2p}} + B_i^{\frac{p + 1}{2p}}] \\
    &= \frac{M \eta^{p + 1}}{16(p + 1)!} \cdot T \cdot ( \Pr(||\nabla F(\hat{x})|| \ge \frac{9M}{8p!} \eta^p)) - T[\frac{(p!)^{\frac{p + 1}{p}}}{4 \sqrt[p]{M}} + 2(\frac{2p!}{M})^{\frac{1}{p }}] \cdot [(\frac{C_1 \cdot \sigma_1^2}{n_1})^{\frac{p + 1}{2p}} + B_1^{\frac{p + 1}{2p}}] \\
    &- T\eta^{\frac{p + 1}{p}}(\frac{(p!)^{\frac{p + 1}{p}}}{4 \sqrt[p]{M}} +  (\frac{2p \cdot p!}{M})^{\frac{1}{p }}) \sum_{i = 2}^p [(\frac{C_i \cdot \sigma_i^2}{n_i})^{\frac{p + 1}{2p}} + B_i^{\frac{p + 1}{2p}}]
    \end{flalign*} which implies that \begin{flalign*}
        &\Pr(||\nabla F(\hat{x})|| \ge \frac{9M}{8p!} \eta^p) \le \frac{16(p + 1)!}{M \eta^{p + 1}T}\Delta + \frac{16(p + 1)!}{M \eta^{p + 1}}[\frac{(p!)^{\frac{p + 1}{p}}}{4 \sqrt[p]{M}} + 2(\frac{2p!}{M})^{\frac{1}{p }}] \cdot [(\frac{C_1 \cdot \sigma_1^2}{n_1})^{\frac{p + 1}{2p}} + B_1^{\frac{p + 1}{2p}}] \\
        &+ \frac{16(p + 1)!}{M \eta^{p + 1}}\eta^{\frac{p + 1}{p}}(\frac{(p!)^{\frac{p + 1}{p}}}{4 \sqrt[p]{M}} +  (\frac{2p \cdot p!}{M})^{\frac{1}{p }}) \sum_{i = 2}^p [(\frac{C_i \cdot \sigma_i^2}{n_i})^{\frac{p + 1}{2p}} + B_i^{\frac{p + 1}{2p}}] \\
    &= \frac{16(p + 1)!}{M \eta^{p + 1}T}\Delta + \frac{16(p + 1)!}{M \eta^{p + 1}}[\frac{(p!)^{\frac{p + 1}{p}}}{4 \sqrt[p]{M}} + 2(\frac{2p!}{M})^{\frac{1}{p }}] \cdot [(\frac{C_1 \cdot \sigma_1^2}{n_1})^{\frac{p + 1}{2p}} + B_1^{\frac{p + 1}{2p}}] \\ 
    &+ \frac{16(p + 1)!}{M \eta^{\frac{p^2 - 1}{p}}} (\frac{(p!)^{\frac{p + 1}{p}}}{4 \sqrt[p]{M}} +  (\frac{2p \cdot p!}{M})^{\frac{1}{p}}) \sum_{i = 2}^p [(\frac{C_i \cdot \sigma_i^2}{n_i})^{\frac{p + 1}{2p}} + B_i^{\frac{p + 1}{2p}}] \\
    \end{flalign*} which finishes the proof. 
\end{proof}

\begin{theorem}
    (Theorem \ref{thm:upperbounddeterministic} restated). For any function $F \in \mathcal{F}_p(\Delta, L_{1:p})$, where $p \ge 2$, $\epsilon > 0$, with biased and stochastic $p^{th}$-order oracles in $\mathcal{O}(F, \sigma_{1:p}, B_{1:p})$ where $\max_{i} B_i \ge  \Omega(\epsilon^{\frac{3p}{3p + 1}})$, with probability at least $\frac{5}{8}$, Algorithm \ref{alg:pth_deterministic} returns a point $\hat{x}$ such that $||\nabla F(\hat{x})|| \le O(\epsilon + \max_{i} B_i)$ and performs at most \begin{flalign*}
        &O(\frac{\Delta (\max_i \sigma_i)^2 }{\epsilon^3 (\max_i B_i + 1)^{\frac{p + 1}{p}}} + \frac{(\epsilon + \max_i B_i)^{\frac{p + 1}{p}}(\max_i \sigma_i)^2}{\epsilon^3 (\max_i B_i + 1)^{\frac{p + 1}{p}}})
    \end{flalign*} queries to the stochastic and biased derivative oracles.
\end{theorem}

\begin{proof}
    From lemma \ref{lemma:gradprobbound}, we have that \begin{flalign*}
        &\Pr(||\nabla F(\hat{x})|| \ge \frac{9M}{8p!} \eta^p) \le \frac{16(p + 1)!}{M \eta^{p + 1}T}\Delta + \frac{16(p + 1)!}{M \eta^{p + 1}}[\frac{(p!)^{\frac{p + 1}{p}}}{4 \sqrt[p]{M}} + 2(\frac{2p!}{M})^{\frac{1}{p }}] \cdot [(\frac{C_1 \cdot \sigma_1^2}{n_1})^{\frac{p + 1}{2p}} + B_1^{\frac{p + 1}{2p}}] \\ 
        &+ \frac{16(p + 1)!}{M \eta^{\frac{p^2 - 1}{p}}} (\frac{(p!)^{\frac{p + 1}{p}}}{4 \sqrt[p]{M}} +  (\frac{2p \cdot p!}{M})^{\frac{1}{p}}) \sum_{i = 2}^p [(\frac{C_i \cdot \sigma_i^2}{n_i})^{\frac{p + 1}{2p}} + B_i^{\frac{p + 1}{2p}}]
    \end{flalign*} Let \begin{flalign*}
        &A = \max(\frac{16(p + 1)!}{M}, \frac{16 (p + 1)!}{M} [\frac{(p!)^{\frac{p + 1}{p}}}{4 \sqrt[p]{M}} + 2(\frac{2p!}{M})^{\frac{1}{p }}], \frac{16(p + 1)!}{M} [\frac{(p!)^{\frac{p + 1}{p}}}{4 \sqrt[p]{M}} +  (\frac{2p \cdot p!}{M})^{\frac{1}{p}}])
    \end{flalign*} Therefore, we have the following upper bound: \begin{flalign*}
        &\frac{A \Delta}{\eta^{p + 1} T} + \frac{A}{\eta^{p + 1}} [(\frac{C_1 \cdot \sigma_1^2}{n_1})^{\frac{p +1}{2p}} + B_1^{\frac{p + 1}{2p}}] + \frac{A}{\eta^{\frac{p^2 - 1}{p}}} \sum_{i = 2}^p [(\frac{C_i \cdot \sigma_i^2}{n_i})^{\frac{p + 1}{2p}} + B_i^{\frac{p + 1}{2p}}]
    \end{flalign*} From our choice of $n_1$ (where such an $n_1$ exists due to $\max_i B_i \ge \Omega(\epsilon^{\frac{3p}{3p + 1}})$ such that \begin{flalign*}
        &\max \{\frac{C_1 \cdot \sigma_1^2}{(\frac{\eta^{p + 1}}{8A} - B_1^{\frac{p + 1}{2p}})^{\frac{2p}{p + 1}}}, 1 \} \le n_1 \le \frac{ (\epsilon + \max_{i} B_i)^{\frac{p + 1}{p}} (\max_i \sigma_i)^2}{\epsilon^3 (\max_{i} B_i + 1)^{\frac{p + 1}{p}}}
    \end{flalign*} we have that \begin{flalign*}
        &\frac{A}{\eta^{p + 1}}[(\frac{C_1 \cdot \sigma_1^2}{n_1}) + B_1^{\frac{p + 1}{2p}}] \le \frac{1}{8}
    \end{flalign*} From our choice of $T = \lceil \frac{8 A \Delta}{\eta^{p + 1}} \rceil$ we have that \begin{flalign*}
        &\frac{A \Delta}{\eta^{p + 1} T} \le \frac{1}{8}
    \end{flalign*} From our choice of $n_i$ (for all $i \ge 2$, which exists due to $\max_i B_i \ge \Omega(\epsilon^{\frac{3p}{3p + 1}}))$ such that \begin{flalign*}
        &\max \{C_i \sigma_i^2 (\frac{\eta^{\frac{p^2 - 1}{p}}}{8Ap} - B_i^{\frac{p + 1}{2p}})^{\frac{-2p}{p + 1}}, 1 \} \le n_i \le \frac{ (\epsilon + \max_i B_i)^{\frac{p + 1}{p}} (\max_i \sigma_i)^2}{\epsilon^3 (\max_{i} B_i + 1)^{\frac{p + 1}{p}}}
    \end{flalign*} it holds that \begin{flalign*}
        &\frac{A}{\eta^{\frac{p^2 - 1}{p}}} [(\frac{C_i \cdot \sigma_i^2}{n_i})^{\frac{p + 1}{2p}} + B_i^{\frac{p + 1}{2p}}] \le \frac{1}{8p}
    \end{flalign*} which implies that \begin{flalign*}
        &\frac{A}{\eta^{\frac{p^2 - 1}{p}}} \sum_{i = 2}^p [(\frac{C_i \cdot \sigma_i^2}{n_i})^{\frac{p + 1}{2p}} + B_i^{\frac{p + 1}{2p}}] \le \frac{(p - 1)}{8p} \le \frac{1}{8}
    \end{flalign*} Therefore, we have that \begin{flalign*}
        &\Pr(||\nabla F(\hat{x})|| \ge \frac{9M}{8p!} (\epsilon + \max_{i} B_i)) \le \frac{3}{8}
    \end{flalign*} which implies that \begin{flalign*}
        &\Pr(||\nabla F(\hat{x})|| < \frac{9M}{8p!} (\epsilon + \max_{i} B_i)) \ge \frac{5}{8}
    \end{flalign*} Now, we give a bound on the oracle complexity. Let $M$ be the total number of oracle queries that we make. In every iteration, we query the $i^{th}$ derivative oracle $n_i$ times which yields that \begin{flalign*}
        &\mathbb{E}[M] =T \sum_{i = 1}^p n_i \\ 
        &\le (\frac{8 A \Delta}{\eta^{p + 1}} + 1) \sum_{i = 1}^p n_i 
    \end{flalign*} Substituting the upper bound for $n_i$, considering the case where $\Omega(\epsilon^{\frac{3p}{3p + 1}}) \le B \le \Theta(1)$, we can further bound this expression by \begin{flalign*}
        &(\frac{8 A \Delta}{\eta^{p + 1}} + 1) \sum_{i = 1}^p \frac{ (\epsilon + \max_j B_j)^{\frac{p + 1}{p}} \cdot (\max_i \sigma_i)^2}{\epsilon^3 (\max_j B_j + 1)^{\frac{p + 1}{p}}} \\ 
        &\le O(\frac{\Delta}{(\epsilon + \max_j B_j)^{\frac{p + 1}{p}}} \cdot \frac{(\epsilon + \max_j B_j)^{\frac{p + 1}{p}}(\max_i \sigma_i)^2}{\epsilon^3 (\max_j B_j + 1)^{\frac{p + 1}{p}}} + \frac{ (\epsilon + \max_j B_j)^{\frac{p + 1}{p}}  (\max_i \sigma_i)^2}{\epsilon^3(\max_j B_j + 1)^{\frac{p + 1}{p}}}) \\ 
        &= O(\frac{\Delta (\max_i \sigma_i)^2}{\epsilon^3 (\max_j B_j + 1)^{\frac{p + 1}{p}}}+ \frac{(\epsilon + \max_j B_j)^{\frac{p + 1}{p}}(\max_i \sigma_i)^2}{\epsilon^3 (\max_j B_j + 1)^{\frac{p + 1}{p}}})
    \end{flalign*} Considering the case where $B > \Omega(1)$, we have that \begin{flalign*}
        &(\frac{8 A \Delta}{\eta^{p + 1}} + 1) \sum_{i = 1}^p \frac{ (\epsilon + \max_j B_j)^{\frac{p + 1}{p}} \cdot (\max_i \sigma_i)^2}{\epsilon^3 (\max_j B_j + 1)^{\frac{p + 1}{p}}} \le O(\frac{\Delta (\epsilon + B)^{\frac{p + 1}{p}}(\max_i \sigma_i)^2}{\epsilon^3 (B + 1)^{\frac{p + 1}{p}}})
    \end{flalign*} since $(1 - \epsilon)^{p + 1}$ can be lower bounded by an $\Theta(1)$ constant. Using Markov's inequality followed by a union bound yields the above oracle complexities with probability at least $\frac{5}{8}$, which finishes the proof.  
\end{proof}

%% file: arxiv.bbl
\begin{thebibliography}{10}

\bibitem{adil2026convex}
Deeksha Adil, Brian Bullins, Arun Jambulapati, and Aaron Sidford.
\newblock Convex optimization with $p$-norm oracles.
\newblock In {\em 37th International Conference on Algorithmic Learning Theory}, 2026.

\bibitem{findingapproxlocalminimafastergd}
Naman Agarwal, Zeyuan Allen-Zhu, Brian Bullins, Elad Hazan, and Tengyu Ma.
\newblock Finding approximate local minima faster than gradient descent.
\newblock {\em Symposium on Theory of Computing}, 2017.

\bibitem{OnConvergenceSGDBiased}
Ahmad Ajalloeian and Sebastian~U. Stich.
\newblock On the convergence of sgd with biased gradients.
\newblock {\em International Conference on Machine Learning, Workshop on "Beyond First Order Methods in ML Systems"}, 2020.

\bibitem{aji2017sparse}
Alham~Fikri Aji and Kenneth Heafield.
\newblock Sparse communication for distributed gradient descent.
\newblock In {\em Proceedings of the 2017 conference on empirical methods in natural language processing}, pages 440--445, 2017.

\bibitem{alistarh2018convergence}
Dan Alistarh, Torsten Hoefler, Mikael Johansson, Nikola Konstantinov, Sarit Khirirat, and C{\'e}dric Renggli.
\newblock The convergence of sparsified gradient methods.
\newblock {\em Advances in Neural Information Processing Systems}, 31, 2018.

\bibitem{allen2018make}
Zeyuan Allen-Zhu.
\newblock How to make the gradients small stochastically: Even faster convex and nonconvex sgd.
\newblock {\em Advances in Neural Information Processing Systems}, 31, 2018.

\bibitem{allen2018natasha}
Zeyuan Allen-Zhu.
\newblock Natasha 2: Faster non-convex optimization than sgd.
\newblock {\em Advances in neural information processing systems}, 31, 2018.

\bibitem{SecondOrderOptNonconvexStochastic}
Yossi Arjevani, Yair Carmon, John~C. Duchi, Dylan~J. Foster, Ayush Sekhari, and Karthik Sridharan.
\newblock Second-order information in non-convex stochastic optimization: Power and limitations.
\newblock {\em Conference on Learning Theory}, 2020.

\bibitem{lowerboundsnonconvexstochastic}
Yossi Arjevani, Yair Carmon, John~C. Duchi, Dylan~J. Foster, Nathan Srebro, and Blake Woodworth.
\newblock Lower bounds for non-convex stochastic optimization.
\newblock {\em Mathematical Programming}, 2019.

\bibitem{bernstein2018signsgd}
Jeremy Bernstein, Yu-Xiang Wang, Kamyar Azizzadenesheli, and Animashree Anandkumar.
\newblock signsgd: Compressed optimisation for non-convex problems.
\newblock In {\em International conference on machine learning}, pages 560--569. PMLR, 2018.

\bibitem{stochgradmethodbiasedestimation}
Jia Bi and Steve~R. Gunn.
\newblock A stochastic gradient method with biased estimation for faster nonconvex optimization.
\newblock {\em arXiv}, 2019.

\bibitem{birgin2017worst}
Ernesto~G Birgin, JL~Gardenghi, Jos{\'e}~Mario Mart{\'\i}nez, Sandra~Augusta Santos, and Ph~L Toint.
\newblock Worst-case evaluation complexity for unconstrained nonlinear optimization using high-order regularized models.
\newblock {\em Mathematical Programming}, 163(1):359--368, 2017.

\bibitem{dimensionfreeacelerationgdconvex}
Yair Carmon, John~C. Duchi, Oliver Hinder, and Aaron Sidford.
\newblock "convex until proven guilty": Dimension-free acceleration of gradient descent on non-convex functions.
\newblock {\em International Conference on Machine Learning}, 2017.

\bibitem{carmon2018accelerated}
Yair Carmon, John~C Duchi, Oliver Hinder, and Aaron Sidford.
\newblock Accelerated methods for nonconvex optimization.
\newblock {\em SIAM Journal on Optimization}, 28(2):1751--1772, 2018.

\bibitem{LowerBoundsI}
Yair Carmon, John~C Duchi, Oliver Hinder, and Aaron Sidford.
\newblock Lower bounds for finding stationary points i.
\newblock {\em Mathematical Programming}, 2019.

\bibitem{LowerBoundsII}
Yair Carmon, John~C Duchi, Oliver Hinder, and Aaron Sidford.
\newblock Lower bounds for finding stationary points ii: First-order methods.
\newblock {\em Mathematical Programming}, 2019.

\bibitem{lion}
Xiangning Chen, Chen Liang, Da~Huang, Esteban Real, Kaiyuan Wang, Yao Liu, Hieu Pham, Xuanyi Dong, Thang Luong, Cho-Jui Hsieh, Yifeng Lu, and Quoc~V. Le.
\newblock Symbolic discovery of optimization algorithms.
\newblock {\em Advances in Neural Information Processing Systems}, 2023.

\bibitem{demidovich2024guide}
Yury Demidovich, Grigory Malinovsky, Igor Sokolov, and Peter Richt{\'a}rik.
\newblock A guide through the zoo of biased sgd.
\newblock {\em Advances in Neural Information Processing Systems}, 36, 2024.

\bibitem{biasedstochgradestimation}
Derek Driggs, Jingwei Liang, and Carola-Bibiane Schönlieb.
\newblock On biased stochastic gradient estimation.
\newblock {\em Journal of Machine Learning Research}, 2022.

\bibitem{dryden2016communication}
Nikoli Dryden, Tim Moon, Sam~Ade Jacobs, and Brian Van~Essen.
\newblock Communication quantization for data-parallel training of deep neural networks.
\newblock In {\em 2016 2nd Workshop on machine learning in hpc environments (MLHPC)}, pages 1--8. IEEE, 2016.

\bibitem{spider}
Cong Fang, Chris Junchi~Li, Zhouchen Lin, and Tong Zhang.
\newblock Spider: Near-optimal non-convex optimization via stochastic path integrated differential estimator.
\newblock {\em Advances in Neural Information Processing Systems}, 2018.

\bibitem{restartedsgd}
Cong Fang, Zhouchen Lin, and Tong Zhang.
\newblock Sharp analysis for nonconvex sgd escaping from saddle points.
\newblock {\em Conference on Learning Theory}, 2019.

\bibitem{stochasticfirstandzerothorder}
Saeed Ghadimi and Guanghui Lan.
\newblock Stochastic first-and zeroth-order methods for nonconvex stochastic programming.
\newblock {\em SIAM Journal on Optimization}, 2013.

\bibitem{hu2016bandit}
Xiaowei Hu, LA~Prashanth, Andr{\'a}s Gy{\"o}rgy, and Csaba Szepesvari.
\newblock (bandit) convex optimization with biased noisy gradient oracles.
\newblock In {\em Artificial Intelligence and Statistics}, pages 819--828. PMLR, 2016.

\bibitem{biasedstochasticfirstordermetalearning}
Yifan Hu, Siqi Zhang, Xin Chen, and Niao He.
\newblock Biased stochastic first-order methods for conditional stochastic optimization and applications in meta learning.
\newblock {\em Advances in Neural Information Processing Systems}, 2020.

\bibitem{adam}
Diederik~P. Kingma and Jimmy Ba.
\newblock Adam: A method for stochastic optimization.
\newblock {\em International Conference on Learning Representations}, 2015.

\bibitem{optmethodstoccompositeopt}
Guanghui Lan.
\newblock An optimal method for stochastic composite optimization.
\newblock {\em Mathematical Programming}, 2011.

\bibitem{li2023restarted}
Huan Li and Zhouchen Lin.
\newblock Restarted nonconvex accelerated gradient descent: No more polylogarithmic factor in the $o(\epsilon^{-7/4})$ complexity.
\newblock {\em Journal of Machine Learning Research}, 24(157):1--37, 2023.

\bibitem{sophia}
Hong Liu, Zhiyuan Li, David Hall, Percy Liang, and Tengyu Ma.
\newblock Sophia: A scalable stochastic second-order optimizer for language model pre-training.
\newblock {\em International Conference on Learning Representations}, 2024.

\bibitem{liu2025muon}
Jingyuan Liu, Jianlin Su, Xingcheng Yao, Zhejun Jiang, Guokun Lai, Yulun Du, Yidao Qin, Weixin Xu, Enzhe Lu, Junjie Yan, et~al.
\newblock Muon is scalable for llm training.
\newblock {\em arXiv preprint arXiv:2502.16982}, 2025.

\bibitem{stochcontrollablebiasedoracles}
Yin Liu and Sam~Davanloo Tajbakhsh.
\newblock Stochastic optimization algorithms for problems with controllable biased oracles.
\newblock {\em arXiv}, 2026.

\bibitem{stacey}
Xinyu Luo, Cedar~Site Bai, Bolian Li, Petros Drineas, Ruqi Zhang, and Brian Bullins.
\newblock Stacey: Promoting stochastic steepest descent via accelerated $\ell_p$-smooth nonconvex optimization.
\newblock {\em International Conference on Machine Learning}, 2025.

\bibitem{matrixconcentration}
Lester Mackey, Michael~I. Jordan, Richard~Y. Chen, Brendan Farrell, and Joel~A. Tropp.
\newblock Matrix concentration inequalities via the method of exchangeable pairs.
\newblock {\em The Annals of Probability}, 2014.

\bibitem{nesterov2012make}
Yurii Nesterov.
\newblock How to make the gradients small.
\newblock {\em Optima. Mathematical Optimization Society Newsletter}, (88):10--11, 2012.

\bibitem{nesterov2017random}
Yurii Nesterov and Vladimir Spokoiny.
\newblock Random gradient-free minimization of convex functions.
\newblock {\em Foundations of Computational Mathematics}, 17(2):527--566, 2017.

\bibitem{stochasticcubicreg}
Nilesh Tripuraneni, Mitchell Stern, Chi Jin, Jeffrey Regier, and Michael~I. Jordan.
\newblock Stochastic cubic regularization for fast nonconvex optimization.
\newblock {\em Advances in Neural Information Processing Systems}, 2017.

\bibitem{zhou2026sharp}
Dongruo Zhou.
\newblock Sharp first-order lower bounds for higher-order smooth nonconvex optimization.
\newblock {\em arXiv preprint arXiv:2606.05438}, 2026.

\bibitem{stochasticnestedvarreduction}
Dongruo Zhou, Pan Xu, and Quanquan Gu.
\newblock Stochastic nested variance reduction for nonconvex optimization.
\newblock {\em Advances in Neural Information Processing Systems}, 2018.

\end{thebibliography}
